\documentclass[12pt,a4paper,oneside]{amsart}
\usepackage{amsmath,amssymb}
\usepackage{amscd}
\usepackage[a4paper]{geometry}
\usepackage{amsthm}
\usepackage{euscript}
\usepackage{latexsym}
\usepackage[matrix,arrow,curve,cmtip]{xy}
\usepackage[cp1251]{inputenc}
\usepackage[english]{babel}
\usepackage{mathrsfs}
\usepackage[pdftex]{graphicx}
\usepackage{keyval}
\usepackage{tikz}
\usetikzlibrary{arrows,shapes,snakes,automata,backgrounds,petri}
\usetikzlibrary{decorations.markings}
\usepackage{verbatim}
\usepackage[symbol*]{footmisc}

\usetikzlibrary{decorations.markings}
\tikzstyle{vertex}=[circle, draw, inner sep=0pt, minimum size=6pt]
\newcommand{\vertex}{\node[vertex]}

\newtheorem{theorem}{Theorem}[section]
\newtheorem{lemma}{Lemma}[section]
\newtheorem{proposition}{Proposition}[section]
\newtheorem{corollary}{Corollary}[section]

\theoremstyle{definition}
\newtheorem{definition}[theorem]{Definition}
\newtheorem{example}[theorem]{Example}
\newtheorem{remark}[theorem]{Remark}
\newtheorem{construction}[theorem]{Construction}

\usepackage{mathptmx} 
\DeclareMathAlphabet{\mathcal}{OMS}{cmsy}{m}{n} 
\usepackage{amssymb,amscd,enumerate,mathrsfs,silence}

\begin{document}

\pagestyle{plain}
\title{\bf BIPARTITE GRAPHS as POLYNOMIALS, and POLYNOMIALS AS BIPARTITE GRAPHS \\
(WITH A VIEW TOWARDS DIVIDING IN $\mathbb{N}[x]$, $\mathbb{N}[x,y]$)}

\maketitle
\begin{center}
Andrey Grinblat\footnote{\texttt{expandrey@mail.ru}} and Viktor Lopatkin\footnote{\texttt{wickktor@gmail.com}, please use this email for contacting.}
\end{center}

\begin{abstract}
  The aim of this paper is to show that any finite undirected bipartite graph can be considered as a polynomial $p  \in \mathbb{N}[x]$, and any directed finite bipartite graph can be considered as a polynomial $p\in\mathbb{N}[x,y]$, and vise verse. We also show that the multiplication in semirings $\mathbb{N}[x]$, $\mathbb{N}[x,y]$ correspondences to a operations of the corresponding graphs which looks like a ``perturbed'' products of graphs. As an application, we give a new point of view to dividing in semirings $\mathbb{N}[x]$, $\mathbb{N}[x,y]$. Finally, we endow the set of all bipartite graphs with the Zariski topology.
\medskip

\textbf{Mathematics Subject Classifications}: 05C20, 05C25, 05C76, 11R09, 16Y60, 68W10, 05C25, 12Y05, 05C31, 13F20, 13B25

\textbf{Key words}: Bipartite graphs; dividing in semirings; Petri nets; polynomials; Winskel's morphisms; the Zariski topology.
\end{abstract}

\section*{Introduction}

In this paper we show that any finite bipartite graph can be considered as a polynomial $p(x) \in \mathbb{N}[x]$, and any finite directed bipartite graph can be considered as a polynomial $p(x,y) \in \mathbb{N}[x,y]$ and vise versa. To be more precisely, we consider any (directed) bipartite graph $\Gamma = (U,V,E)$ with an injection map $\varphi:V \to \mathbb{N}$ on the set $V$ of its vertices, \textit{i.e.,} we label every vertex $v\in V$ by some natural number. For every pair $(\Gamma, \varphi)$ we construct (see Constructions \ref{congraphpol}) a polynomial $p(N,\varphi) \in \mathbb{N}[x]$ and a polynomial $p(\Gamma,\varphi) \in \mathbb{N}[x,y]$ (see Construction \ref{construction1})) if $\Gamma$ is directed. Next, we present an inverse procedure (see Construction \ref{conpolgraph}), namely, for any polynomial $p(x) \in \mathbb{N}[x]$ we construct a bipartite graph $\Gamma(p) = (U,V,E)$ with a ``natural'' injection $\iota: V \to \mathbb{N}$ and we prove that $\Gamma \cong \Gamma(p(\Gamma,\varphi))$ (see Proposition \ref{=}). Furthermore, we show that for every two polynomials $p_1,p_2 \in \mathbb{N}[x]$ their multiplication $p_1\cdot p_2$ correspondences to a an operation of the corresponding graphs $\Gamma(p_1), \Gamma(p_2)$, we call  it \textit{polynomial product} and we denote it by $\Gamma_1\times_{\varphi_1,\varphi_2}\Gamma_2$. A graph $\Gamma_1\times_{\varphi_1,\varphi_2}\Gamma_2$ looks like the product $\Gamma_1\times \Gamma_2$ ``perturbed'' by $\mathrm{Im}(\varphi_1)\cap \mathrm{Im}(\varphi_2)$ (see Definition \ref{polynprodofgraph}, Proposition \ref{product}). Further, we shall also see that the sum $p+q$ correspondences to a graph $\Gamma$ that can be obtained from $\Gamma(p_1)$ and $\Gamma(p_2)$ by ``attaching to each other'' by its natural injections (see Definition \ref{sumofgraph} and Proposition \ref{propsumofgraph}). Similar results for directed bipartite graphs are obtained in Sections 4,5.

Finally, all of this enables us to consider dividing in semirings $\mathbb{N}[x]$, $\mathbb{N}[x,y]$ via (directed) bipartite graph point of view (see Section 6 and Example \ref{ex}), and to introduce the Zariski topology on the set of all finite (directed) bipartite graphs. With respect to this topology, in particular, we get an one-to-one correspondence between the set of all prime ideals of $\mathbb{N}[x]$ (\textit{resp.} $\mathbb{N}[x,y]$) and the set of all irreducible bipartite graphs (Theorem \ref{criteriafordiv}).

\vspace{2ex}

{\bf Acknowledgements:} the authors would like to express their deepest gratitude to {\sc Prof. Ernst W. Mayr}, who has drawn the author's attention to the graph point of view of some Petri nets problems studied by authors before this work, and for having kindly clarified some very important details. Special thanks are due to {\sc Prof. Glynn Winskel} for useful discussions. We are also grateful to {\sc Prof. Ilias S. Kotsireas}.

\section{Preliminaries}
Here we recall some definitions and notations that will be frequently used.

\begin{definition}
An {\it isomorphism of graphs} $\Gamma_1 = (V_1,E_1)$ and $\Gamma_2 = (V_2,E_2)$ is a bijection between the sets $V_1$ and $V_2$ $f:V_1 \to V_2$ such that any two vertices $v,u$ of $\Gamma_1$ are adjacent in $\Gamma$ if and only if $f(v)$ and $f(u)$ are adjacent in $\Gamma_2$.
\end{definition}

Recall that a {\it bipartite graph} or {\it bigraph} is a graph whose vertices can be divided into two disjoint and independent sets $V,U$ such that every edge connects a vertex in $V$ to one in $U$. Vertex sets $V$ and $U$ are usually called the {\it parts} of the graph.

We often write $\Gamma = (U,V;E)$ to denote a bipartite graph whose partition has the parts $V$ and $U$, with $E$ denoting the edges of the graph. Further, we say that $(V,U)$ is a {\it bipartition} of $\Gamma$.

\begin{lemma}{\cite[Property 2.1.1]{ADH}}\label{prop1}
  A connected bipartite graph has a unique bipartition.
\end{lemma}

The definition of an isomorphism of graphs and the previous lemma imply the following

\begin{lemma}
   Let $\Gamma_1 = (V_1,U_1,E_1)$, $\Gamma_2 = (V_2,U_2,E_2)$ be connected bipartite graphs. Then they are isomorphic if and only if there exist bijections either $\beta_1: V_1 \to V_2$, $\beta_2:U_1 \to U_2$ or $\beta_1: V_1 \to U_2$, $\beta_2:U_1 \to V_2$ such that any two vertices $v_1 \in V_1$, $u_1\in U_1$ are adjacent in $\Gamma_1$ if and only if $\beta_1(v_1)$ and $\beta_2(u_1)$ are adjacent in $\Gamma_2$.
 \end{lemma}
\begin{proof}
  Indeed, let $f:\Gamma_1 \to \Gamma_2$ be an isomorphism. It is clear that $(f(V_1), f(U_1))$ is a bipartition of $\Gamma_2$. Using Lemma \ref{prop1} we complete the proof.
\end{proof}

Thus, without loss of generality, we can say that {\it an isomorphism between two connected bipartite graphs $\Gamma_1 = (V_1,U_1,E_1)$ and $\Gamma_2 = (V_2,U_2,E_2)$ is a pair $(\beta_V,\beta_U)$ of bijections $\beta_V:V_1 \to V_2$, $\beta_U:U_1 \to U_2$ such that any two vertices $v_1 \in V_1$, $u_1\in U_1$ are adjacent in $\Gamma_1$ if and only if $\beta_V(v_1)$ and $\beta_U(u_1)$ are adjacent in $\Gamma_2$.} We shall then write
\[
(\beta_V,\beta_U): \Gamma_1 = (V_1,U_1,E_1) \cong \Gamma_2 = (V_2,U_2,E_2).
\]

We shall also deal with bipartite directed  graphs. Recall that a \textit{directed graph} is a graph whose edges have been oriented. We sometimes call these edges \textit{arcs} or \textit{directed edges}, and when referring to the endpoints of an arc, say an arc is directed from head to tail. For an arc $a$, we denote by $\mathfrak{h}(a)$ and $\mathfrak{t}(a)$ its head and tail respectively, and we thus can write $a = (\mathfrak{h}(a),\mathfrak{t}(a))$.

Thus, a \textit{bipartite directed graph} is a directed graph whose vertices can be divided into two disjoin and independent sets $U,V$ such that every arc connects a vertex in $V$ to one in $U$. Vertex sets $U$, $V$ are also called the parts of the graph.

We often write $\vec \Gamma = (U,V;A)$ to denote a bipartite directed graph whose partition has the parts $U$ and $V$, with $A$ denoting the arcs of the graph.

Similarly one can easy to define an isomorphism of two connected bipartite graphs. Namely, take two connected bipartite directed graphs $\vec \Gamma_i = (U_i,V_i;A_i)$, $i=1,2$. We sat that \textit{an isomorphism between them is a pair $(\beta_U,\beta_V)$ of bijections $\beta_U:U_1 \to U_2,$ $\beta_V: V_1 \to V_2$ such that $(u,v) \in A_1$ if and only if $(\beta_U(u),\beta_V(v)) \in A_2$.}

\begin{definition}\label{definitionoftau}
Let $k= \varepsilon_0(k)2^0 +\varepsilon_1(k)2^1 + \cdots + \varepsilon_{\ell_k}(k)2^{\ell_k}$ be a binary decomposition of $k \in \mathbb{N}$. Set $\tau(k): = \{t:\varepsilon_t(k) =1\}$ if $k>0$ and $\tau(0) = \varnothing$ otherwise. It is obviously that $\tau(k) = \tau(k')$ implies that $k = k'$ and vise versa.

Take a polynomial $p = \sum_{s=1}^\ell\sum_{d_{i_1},\ldots, d_{i_s}} n_{d_{i_1},\ldots, d_{i_s}} x_{i_1}^{d_{i_1}}\cdots x_{i_s}^{d_{i_s}}  \in \mathbb{N}[x_1,\ldots, x_m]$ and set
\[
 \tau(p):=\bigcup_{s=1}^\ell\bigcup_{d \in \{d_{i_1},\ldots, d_{i_s}\}}\{\tau(d) \}.
\]
\end{definition}

\begin{lemma}\label{tau}
Let $p = p_1 \cdot p_2 \in \mathbb{N}[x_1,\ldots, x_m]$; then $\tau(p_1) \cap \tau(p_2) = \varnothing$ if and only if $\tau(p) = \tau(p_1) \cup \tau(p_2)$.
\end{lemma}
\begin{proof}
It suffices to prove that for any two integers $i_1,i_2 \in \mathbb{N}$, $\tau(i_1) \cap \tau(i_2) = \varnothing$ if and only if  $\tau(i_1+i_2) = \tau(i_1) \cup \tau(i_2)$, and then the statement follows rather easily.

Let $i_1 = \varepsilon_0^{(1)}2^0 + \varepsilon_1^{(1)}2^1 + \ldots +\varepsilon_0^{(1)}2^{\ell_{i_1}}$ and $i_2 = \varepsilon_0^{(2)}2^0 + \varepsilon_1^{(2)}2^1 + \ldots +\varepsilon_0^{(2)}2^{\ell_{i_2}}$. Assume that $\tau(i_1)\cap \tau(i_2) = \varnothing$. It follows that $i_1+i_2 = \varepsilon_0 2^0 + \varepsilon_12^1 + \ldots +\varepsilon_\ell2^{\ell}$, where $\varepsilon_j = \mathrm{max}\{\varepsilon_j^{(1)},\varepsilon_j^{(2)}\}$, $1 \le j \le \mathrm{max}\{\ell _{i_1}, \ell_{i_2}\} = \ell,$ i.e., $\tau(i_1 + i_2) = \tau(i_1) \cup \tau(i_2).$ Conversely, let $\tau(i_1 + i_2) = \tau(i_1) \cup \tau(i_2)$ and $\tau(i_1) \cap \tau(i_2) \ne \varnothing$. We have $\varepsilon_n^{(1)} + \varepsilon_n^{(2)} = 1+1 = 0\bmod (2)$ for every $n \in \tau(i_1) \cap \tau(i_2)$. This implies that $n \notin \tau(i_1+i_2)$. But, on the other hand, we have $n \in \tau(i_1) \cap \tau(i_2) \subseteq \tau(i_1) \cup \tau(i_2) =  \tau(i_1+i_2)$, a contradiction that completes the proof.
\end{proof}

\section{Bipartite Graph Polynomials}
In this section we introduce main constructions and we prove some properties of them. Further we consider products of bipartite graph in the context of the constructions.

\subsection{Main Constructions and Properties}
Given a bipartite graph $\Gamma = (U,V;E)$, where $U\cup V$ is a set of its vertices. We are going to construct a polynomial $p(\Gamma) \in \mathbb{N}[x]$ and show that every polynomial from the semiring $p(x) \in \mathbb{N}[x]$ gives a bipartite graph $\Gamma(p)$. We then show that $\Gamma(P(\Gamma)) \cong \Gamma$ for every bipartite graph $\Gamma$.

\begin{construction}\label{congraphpol}
  Let $\Gamma = (V,U;E)$ be a bipartite graph and $\varphi: V \to \mathbb{N}$ an injection. Set $P(\Gamma,\varphi):=\sum_{u \in U}x^{d(u)} \in \mathbb{N}[x]$, where $d(u):=\sum_{v \in V\cap E(u)} 2^{\varphi(v)}$, if $E(u) = \varnothing$ we put $d(u) = 1$.
\end{construction}

\begin{example}
Let us consider the following bipartite graph $\Gamma$.
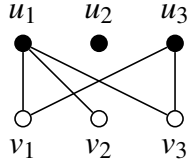
\[
 \begin{tikzpicture}[semithick]
	\vertex[fill] (u1) at (0,1) [label=above:$u_1$] {};
  \vertex[fill] (u2) at (1,1) [label=above:$u_2$] {};
	\vertex[fill] (u3) at (2,1) [label=above:$u_3$] {};
	\vertex (v1) at (0,0) [label=below:$v_{1}$] {};
	\vertex (v2) at (1,0) [label=below:$v_{2}$] {};
	\vertex (v3) at (2,0) [label=below:$v_{3}$] {};
	\path
		(u1) edge (v1)
		(u1) edge (v2)
		(u1) edge (v3)
		(u3) edge (v1)
		(u3) edge (v3)
		;
\end{tikzpicture}\]

Let $\varphi(v_1) = 0$, $\varphi(v_2) = 1$ and $\varphi(v_3) = 2$. We then have

\[
   d(u_1) = x^{2^0 + 2^1 + 2^2} = x^7, \quad d(u_2) = 1, \quad d(u_3) = x^{2^0 + 2^2} = x^5.
\]
Thus, $P(\Gamma,\varphi) = x^7 + x^5 +1$.

\end{example}

\begin{remark}
  It is easy to see that if a bipartite graph $\Gamma = (V,U,E)$ has isolated vertices belong to $U$ then $p(x) \ne 0$, where $p(x):=P(\Gamma,\varphi)$ is its polynomial. Moreover it is also easy to see that if $\Gamma$ has $n$ connected components $\Gamma_1, \ldots, \Gamma_n$ then its polynomial is a sum of the corresponding polynomials of its components. We shall consider this case more clearly later.
\end{remark}

\begin{proposition}
  Let $\Gamma = (V,U,E)$, $\Gamma' = (V',U',E')$ be two connected bipartite graphs; if there is an isomorphism $(\beta_V,\beta_U): \Gamma \xrightarrow{\cong} \Gamma'$, then $P(\Gamma,\varphi) = P(\Gamma', \varphi \circ \beta_V^{-1})$, where $\varphi: V \to \mathbb{N}$ is an arbitrary injection.
\end{proposition}

Roughly speaking two connected bipartite graphs are isomorphic if and only if they have the same polynomial.

\begin{proof}
  Let $P(\Gamma,\varphi) = \sum_{u \in U} x^{i(e)}$ and $\beta_V(v) = v'$, $\beta_U(u)=u'$, for every $v \in V$, $u\in U$. We have
\begin{eqnarray*}
    d(e) &: =& \sum_{v \in V \cap E(u)}2^{\varphi(v)} = \sum_{\beta^{-1}_V(v') \in V\cap E (\beta_U^{-1}(u))} 2^{\varphi(\beta_V^{-1}(v'))} \\
    &=& \sum_{v' \in V'\cap E(u')} 2^{\varphi(\beta_V^{-1}(v'))} =: d(e'),
\end{eqnarray*}
where $e' \in E'$. Hence $P(\Gamma,\varphi) = P(\Gamma', \varphi \circ \beta_V^{-1})$, as claimed.
\end{proof}

So, we have constructed a polynomial $P\in\mathbb{N}[x]$ for every bipartite graph. Now, we want to present an inverse procedure. We use Definition \ref{definitionoftau}.

\begin{construction}\label{conpolgraph}
Let $I\subset \mathbb{N}$ be a finite subset. Let $P = P(x) = \sum_{i \in I} n_{i}x^i\in \mathbb{N}[x]$ be a polynomial. Take $i \in I$ and consider a set $U^{i}:=\{u_1^{i},\ldots,u_{n_{i}}^{i} \}$ of some elements. Construct a bipartite graph $\Gamma(P): = (V,U,E)$, where
\[
 V :=\tau(P):= \bigcup_{i \in I} \tau(i), \quad U := \bigcup_{i \in I} U^i, \quad E:= \{(\tau(i), u^i_k): i \in I, k = 1,\ldots, n_i\},
 \]
and an injection $\iota:V \to \mathbb{N}$, which we call  \textit{the natural injection}, are defined as $\iota(\tau(i)) = \tau(i).$ Isolated vertices of this graph belong to $U$ and have a form $(\varnothing, e_k^i)$.
\end{construction}

\begin{example}
  Take a polynomial $p(x) = 2x^5 + x^3 + x^2 + 2$. We have $I = \{5,3,2,0\}$. Since $5 = 2^2 + 2^0$, $3= 2^1 +2^0$, $2 = 2^1$ we then have
  \[
   \tau(5) = \{2,0\}, \quad \tau(3) = \{1,0\}, \quad \tau(2) = \{1\}, \quad \tau(0) := \{\varnothing\},
  \]
hence $V = \{2,1,0\}$.

Next,
\[
 U^5 = \{u_1^5,u_2^5\}, \quad U^3 = \{u_1^3\}, \quad U^2 = \{u_1^2\}, \quad U^0 = \{u_1^0,u_2^0\},
\]
and we thus obtain the following bipartite graph
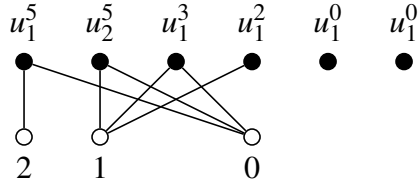
\[\begin{tikzpicture}[semithick]
	\vertex[fill] (u15) at (0,1) [label=above:$u_1^5$] {};
  \vertex[fill] (u25) at (1,1) [label=above:$u_2^5$] {};
	\vertex[fill] (u13) at (2,1) [label=above:$u_1^3$] {};
  \vertex[fill] (u12) at (3,1) [label=above:$u_1^2$] {};
  \vertex[fill] (u10) at (4,1) [label=above:$u_1^0$] {};
	\vertex[fill] (u20) at (5,1) [label=above:$u_1^0$] {};
	\vertex (2) at (0,0) [label=below:$2$] {};
	\vertex (1) at (1,0) [label=below:$1$] {};
	\vertex (0) at (3,0) [label=below:$0$] {};
	\path
		(u15) edge (2)
		(u15) edge (0)
		(u25) edge (1)
		(u25) edge (0)
		(u13) edge (1)
    (u13) edge (0)
    (u12) edge (1)
 	;
\end{tikzpicture}\]
\end{example}

\begin{proposition}\label{=}
  Let $\Gamma = (V,U,E)$ be a bipartite graph and $\varphi:V \to \mathbb{N}$ an injection. The graphs $\Gamma$ and $\Gamma(P(\Gamma,\varphi))$ are isomorphic.
\end{proposition}
\begin{proof}
Let $\Gamma(P(\Gamma,\varphi)) =(\widetilde{V},\widetilde{U},\widetilde{E})$. Fix $u\in U$ and let $\{v_{1,u}, \ldots,v_{k_u,u}\} \subseteq V$ be a set of all vertices are connected with $u$. Thus $d(u) = 2^{\varphi(v_{1,u})} + \cdots + 2^{\varphi(v_{k_u,u})}$, and hence $\tau(d(u)) = \{\varphi(v_{1,u}),\ldots, \varphi(v_{k_u,u})\}$, {\it i.e.,} $\widetilde{V}:=\tau(P(\Gamma,\varphi)) = \mathrm{Im}(\varphi)$.

Next, let $p(\Gamma,\varphi) = \sum n_u x^{d(u)}$, by Construction \ref{conpolgraph}, $\widetilde{U} = \cup_{u \in U} \widetilde{U}^{d(u)}$, where $\widetilde{U}^{d(u)} = \{u_1^{d(u)}, \ldots, u_{n_u}^{d(u)}\}$.

We thus get $\widetilde{E} = \cup_{u \in U} \cup_{1 \le p \le k_u} \cup_{1 \le q \le n_u} \{(\varphi(v_{p,u}),u_q^{d(u)})\}$. It is obviously that for a fixed $q$, the set $\{\varphi(v_{1,u}), \ldots, \varphi(v_{k_u,u})\}$ one to one correspondences to the set $\{v_{1_u}, \ldots, v_{k_u,u}\}$. Further, from Construction \ref{congraphpol} it follows that $\sum \alpha_u = |U|$, hence $|\widetilde{U}| = \sum  \alpha_u = |U|$. Thus there exist bijections between $U$ and $\widetilde{U}$. Finally, since $\varphi$ is an injection we complete the proof.
\end{proof}

\section{Algebraic Operations on Bipartite Graphs}

\subsection{Polynomial Product of Bigraphs} In this subsection we introduce a binary operation on labeled bipartite graphs. We call this operation \textit{polynomial product} because of it is defined by the product of its polynomials. This operation, as we can see below, is a generalization of the usual product of graphs.

We start with the following examples.

\begin{example}
  Let us consider the following bigraphs $\Gamma_1 = (V_1,U_1,E_1)$ and $\Gamma_2 = (V_2,U_2,E_2)$, where $V_1 = \{v_{11}, v_{12}\}$, $U_1 = \{u_{11}, u_{12}\}$, $V_2 = \{v_{21}, v_{22}\}$, $U_2 = \{u_{21}, u_{22}\},$ (see the figure below):
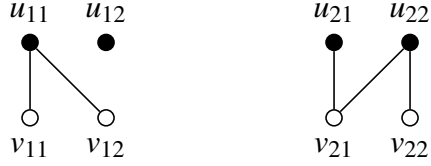
\[
 \begin{tikzpicture}[semithick]
	\vertex[fill] (u11) at (0,1) [label=above:$u_{11}$] {};
  \vertex[fill] (u12) at (1,1) [label=above:$u_{12}$] {};
  \vertex[fill] (u21) at (4,1) [label=above:$u_{21}$] {};
  \vertex[fill] (u22) at (5,1) [label=above:$u_{22}$] {};
	\vertex (v11) at (0,0) [label=below:$v_{11}$] {};
	\vertex (v12) at (1,0) [label=below:$v_{12}$] {};
  \vertex (v21) at (4,0) [label=below:$v_{21}$] {};
  \vertex (v22) at (5,0) [label=below:$v_{22}$] {};
	\path
		(u11) edge (v11)
		(u11) edge (v12)
    (u21) edge (v21)
    (u22) edge (v21)
    (u22) edge (v22)
		;
\end{tikzpicture}
\]

Let $\varphi_1: V_1 \to \mathbb{N}$, $\varphi_2:V_2 \to \mathbb{N}$ be given as follows:
\[
 \varphi_1(v_{11}) = 0, \varphi_1(v_{12}) = 1, \varphi_2(v_{21}) = 2, \varphi_2(v_{22}) = 3,
\]
then $p(\Gamma_1,\varphi_1) = x^3 +1$, and $p(\Gamma_2,\varphi_2) = x^4 + x^{12}$.

Consider the polynomial $p = p(\Gamma_1,\varphi_1) \cdot p(\Gamma_2,\varphi_2)$,
\[
p = x^{15} + x^{12} +x^7 + x^4.
\]

We obtain $I = \{15,12,7,4\}$, $\tau(p) = \{3,2,1,0\}$ because of
\[
 15 = 2^3 + 2^2 + 2^1 + 2^0, \, 12 = 2^3 + 2^2, \, 7 = 2^2 + 2^1 + 2^0, \, 4 = 2^2,
\]
and $U = \{u_1^{15}, u_1^{12}, u_1^7, u_1^4\}$, and hence $\Gamma(p)$ has the following form
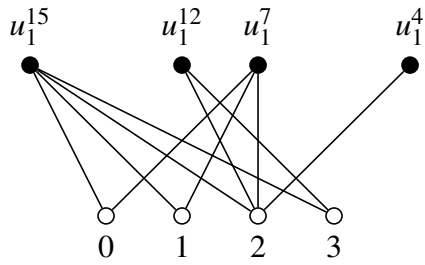
\[
 \begin{tikzpicture}[semithick]
	\vertex[fill] (u115) at (-1,2) [label=above:$u_{1}^{15}$] {};
  \vertex[fill] (u112) at (1,2) [label=above:$u_{1}^{12}$] {};
  \vertex[fill] (u17) at (2,2) [label=above:$u_{1}^{7}$] {};
  \vertex[fill] (u14) at (4,2) [label=above:$u_{1}^4$] {};
	\vertex (0) at (0,0) [label=below:$0$] {};
	\vertex (1) at (1,0) [label=below:$1$] {};
  \vertex (2) at (2,0) [label=below:$2$] {};
  \vertex (3) at (3,0) [label=below:$3$] {};
	\path
		(u115) edge (0)
    (u115) edge (1)
    (u115) edge (2)
    (u115) edge (3)
    (u112) edge (3)
		(u112) edge (2)
    (u17) edge (2)
    (u17) edge (1)
    (u17) edge (0)
    (u14) edge (2)
		;
\end{tikzpicture}
\]

It is easy to see that $\Gamma(p) \cong \Gamma_1 \times \Gamma_2$. Indeed,by
\[
 u_1^{15} \mapsto (u_{11},u_{22}), \, u_1^{12} \mapsto (u_{12},u_{22}), \, u_1^7 \mapsto (u_{11},u_{21}), \, u_1^4 \mapsto (u_{12}, u_{21}),
\]
the isomorphism is clear.

\hfill$\square$
\end{example}

\begin{example}
Let us consider the same graphs as before
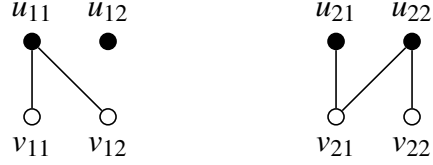
\[
 \begin{tikzpicture}[semithick]
	\vertex[fill] (u11) at (0,1) [label=above:$u_{11}$] {};
  \vertex[fill] (u12) at (1,1) [label=above:$u_{12}$] {};
  \vertex[fill] (u21) at (4,1) [label=above:$u_{21}$] {};
  \vertex[fill] (u22) at (5,1) [label=above:$u_{22}$] {};
	\vertex (v11) at (0,0) [label=below:$v_{11}$] {};
	\vertex (v12) at (1,0) [label=below:$v_{12}$] {};
  \vertex (v21) at (4,0) [label=below:$v_{21}$] {};
  \vertex (v22) at (5,0) [label=below:$v_{22}$] {};
	\path
		(u11) edge (v11)
		(u11) edge (v12)
    (u21) edge (v21)
    (u22) edge (v21)
    (u22) edge (v22)
		;
\end{tikzpicture}
\]
but we define $\varphi_1: V_1 \to \mathbb{N}$, $\varphi_2:V_2 \to \mathbb{N}$ as follows:
\[
 \varphi_1(v_{11}) = 0, \varphi_1(v_{12}) = 1, \varphi_2(v_{21}) = 1, \varphi_2(v_{22}) = 2,
\]
{\it i.e.,} $\mathrm{Im}(\varphi_1) \cap \mathrm{Im}(\varphi_2) \ne \varnothing$.

We obtain $p(\Gamma_1, \varphi_1) = x^3 +1$, and $p(\Gamma_2,\varphi_2) = x^6 + x^2$. We get
\[
 p = p(\Gamma_1, \varphi_1)\cdot p(\Gamma_2, \varphi_2) = x^9 + x^6 + x^5 +x^2.
\]

Then $I = \{9,6,5,2\}$, $\tau(p) = \{3,2,1,0\}$, and $U = \{u^9_1,u^6_1,u^5_1,u^2_1\}$, and hence $\Gamma(p)$ has the following form
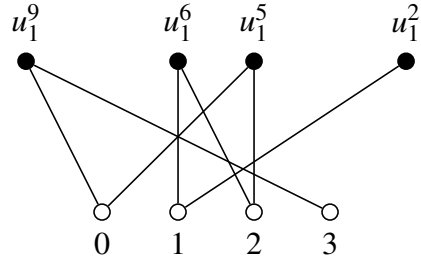
\[
 \begin{tikzpicture}[semithick]
	\vertex[fill] (u19) at (-1,2) [label=above:$u_{1}^{9}$] {};
  \vertex[fill] (u16) at (1,2) [label=above:$u_{1}^{6}$] {};
  \vertex[fill] (u15) at (2,2) [label=above:$u_{1}^{5}$] {};
  \vertex[fill] (u12) at (4,2) [label=above:$u_{1}^2$] {};
	\vertex (0) at (0,0) [label=below:$0$] {};
	\vertex (1) at (1,0) [label=below:$1$] {};
  \vertex (2) at (2,0) [label=below:$2$] {};
  \vertex (3) at (3,0) [label=below:$3$] {};
	\path
		(u19) edge (0)
    (u19) edge (3)
    (u16) edge (2)
		(u16) edge (1)
    (u15) edge (2)
    (u15) edge (0)
    (u12) edge (1)
    ;
\end{tikzpicture}
\]

To understand how this graph can be obtained by $\Gamma_1$, $\Gamma_2$, let us consider the product $p=p(\Gamma_1,\varphi_1) \cdot p(\Gamma_2,\varphi_2)$ more precisely. We have
\begin{eqnarray*}
  p &=& (x^{2^1 + 2^0} + x^0)\cdot (x^{2^1} + x^{2^1 + 2^2}) \\
  &=& x^{2^1 + 2^0+ 2^1} + x^{2^1 + 2^0 + 2^1 +2^2} + x^{2^1} + x^{2^1 + 2^2}\\
  &=& x^{2^2 + 2^0} + x^{2^3 + 2^0} + x^{2^1} + x^{2^1 + 2^2}.
\end{eqnarray*}

Set
\[
 u_1^9 \mapsto (u_{11}, u_{21}), \quad u_1^6 \mapsto (u_{12},u_{21}), \quad u_1^5 \mapsto (u_{11}, u_{22}), \quad u_1^2 \mapsto (u_{12}, u_{22}).
\]
\

Hence $\Gamma(p)$ can be reimaged as follows
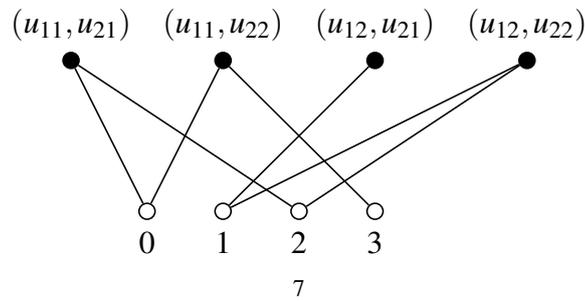
\[
 \begin{tikzpicture}[semithick]
	\vertex[fill] (u1) at (-1,2) [label=above:$\mbox{$(u_{11},u_{21})$}$] {};
  \vertex[fill] (u2) at (1,2) [label=above:$\mbox{$(u_{11},u_{22})$}$] {};
  \vertex[fill] (u3) at (3,2) [label=above:$\mbox{$(u_{12},u_{21})$}$] {};
  \vertex[fill] (u4) at (5,2) [label=above:$\mbox{$(u_{12},u_{22})$}$] {};
	\vertex (0) at (0,0) [label=below:$0$] {};
	\vertex (1) at (1,0) [label=below:$1$] {};
  \vertex (2) at (2,0) [label=below:$2$] {};
  \vertex (3) at (3,0) [label=below:$3$] {};
	\path
		    (u1) edge (0)
        (u1) edge (2)
        (u2) edge (0)
        (u2) edge (3)
        (u3) edge (1)
        (u4) edge (1)
        (u4) edge (2)
    ;
\end{tikzpicture}
\]

\hfill$\square$
\end{example}

We see that the graph $\Gamma(p)$ looks like the product $\Gamma_1 \times \Gamma_2$ ``perturbed'' by $\mathrm{Im}(\varphi_1) \cap \mathrm{Im}(\varphi_2)$. To be more precisely, we introduce a binary operation which generalizes the product of graphs.

\begin{definition}\label{polynprodofgraph}
  Given two bigraphs $\Gamma_i = (U_i,V_i,E_i)$ with injections $\varphi_i:V_i \to \mathbb{N}$, $i = 1,2$. Their \textit{polynomial product} $\Gamma_1 \times_{\varphi_1,\varphi_2} \Gamma_2$, is defied as follows
  \[
  \Gamma_1 \times_{\varphi_1,\varphi_2} \Gamma_2 := \Gamma(p(\Gamma_1,\varphi_1)\cdot p(\Gamma_2,\varphi_2)).
  \]
\end{definition}

\begin{proposition}\label{product}
 Let $\Gamma_1 = (V_1,U_1,E_1)$, $\Gamma_2 = (V_2,U_2,E_2)$ be two bigraphs and let $\varphi_1:V_1\to \mathbb{N}$, $\varphi_2:V_2\to \mathbb{N}$ be two arbitrary injections. Then $\Gamma_1 \times_{\varphi_1,\varphi_2} \Gamma_2 = (V,U,E)$, where:
 \begin{align*}
   & V= \bigcup\nolimits_{u_1 \in U_1, u_2 \in U_2} \{\tau(d(u_1) +d(u_2))\} , \\
   & U = U_1\times U_2, \\
   & E = \bigcup\nolimits_{\substack{u_1 \in U_1, u_2 \in U_2 \\ n \in \tau( d(u_1) + d(u_2)) }} \left\{ \bigl( (u_1,u_2),n \bigr) \right\}
 \end{align*}
\end{proposition}
\begin{proof} The proof is immediately follows from Construction \ref{congraphpol}, product of polynomials, and Construction \ref{conpolgraph}.
\end{proof}

\begin{corollary}
  If $\mathrm{Im}(\varphi_1) \cap \mathrm{Im}(\varphi_2) =\varnothing$ then $\Gamma_1 \times_{\varphi_1,\varphi_2}\Gamma_2 = \Gamma_1 \times \Gamma_2$.
\end{corollary}
\begin{proof}
  Indeed if $\mathrm{Im}(\varphi_1) \cap \mathrm{Im}(\varphi_2) =\varnothing$ then by Lemma \ref{tau}, the statement follows.
\end{proof}

\subsection{Polynomial Coproduct of Bigraphs} Here we show that sum of two graph polynomials correspondences to a simple operations on labeled bigraphs.

\begin{example}
  Let us consider the following bigraphs $\Gamma_1 = (U_1,V_1,E_1)$, $\Gamma_2 = (U_2,V_2,E_2)$ which are shown below
  
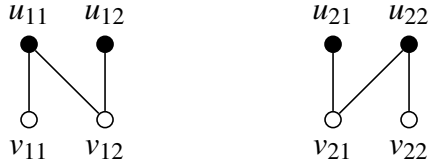
\[
  \begin{tikzpicture}[semithick]
	\vertex[fill] (u11) at (0,1) [label=above:$u_{11}$] {};
  \vertex[fill] (u12) at (1,1) [label=above:$u_{12}$] {};
  \vertex[fill] (u21) at (4,1) [label=above:$u_{21}$] {};
  \vertex[fill] (u22) at (5,1) [label=above:$u_{22}$] {};
	\vertex (v11) at (0,0) [label=below:$v_{11}$] {};
	\vertex (v12) at (1,0) [label=below:$v_{12}$] {};
  \vertex (v21) at (4,0) [label=below:$v_{21}$] {};
  \vertex (v22) at (5,0) [label=below:$v_{22}$] {};
	\path
		(u11) edge (v11)
		(u11) edge (v12)
    (u12) edge (v12)
    (u21) edge (v21)
    (u22) edge (v21)
    (u22) edge (v22)
		;
\end{tikzpicture}
\]

Define injections $\varphi_1: V_1 \to \mathbb{N}$, $\varphi_2:V_2 \to \mathbb{N}$ as follows
\[
 \varphi_1(v_{11}) = 0, \quad \varphi(v_{12}) = 1, \quad \varphi(v_{21}) = 1, \quad \varphi(v_{22}) = 2.
\]

We obtain
\begin{align*}
  & d(u_{11}) = 2^0 + 2^1 = 3, \quad d(u_{12}) = 2^1 = 2,\\
  & d(u_{21}) = 2^1, \quad d(u_{22}) = 2^1 + 2^2 = 6,
\end{align*}
hence
\[
 p_1 := p(\Gamma_1,\varphi_1) = x^2 + x^3, \quad p_2:=p(\Gamma_2,\varphi_2) = x^6 + x^2,
\]
then
\[
 p:=p_1 +p_2 = x^6 + x^3 + 2x^2.
\]

Construct $\Gamma(p)$. We get
\[
 I= \{6,3,2\}, \quad \tau(p) = \{0,1,2\}, \quad U = \{u_1^6, u_1^3,u_1^2,u_2^2\}.
\]

Then the graph $\Gamma(p)$ can be imaged as follows
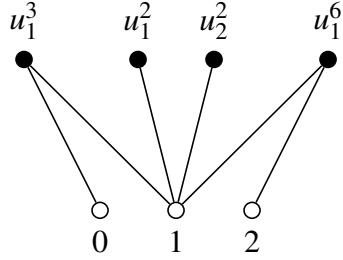
\[
 \begin{tikzpicture}[semithick]
	\vertex[fill] (u1) at (-1,2) [label=above:$u_1^3$] {};
  \vertex[fill] (u2) at (0.5,2) [label=above:$u_1^2$] {};
  \vertex[fill] (u3) at (1.5,2) [label=above:$u_2^2$] {};
  \vertex[fill] (u4) at (3,2) [label=above:$u_1^6$] {};
	\vertex (0) at (0,0) [label=below:$0$] {};
	\vertex (1) at (1,0) [label=below:$1$] {};
  \vertex (2) at (2,0) [label=below:$2$] {};
	\path
		    (u1) edge (0)
        (u1) edge (1)
        (u2) edge (1)
        (u3) edge (1)
        (u4) edge (1)
        (u4) edge (2)
    ;
\end{tikzpicture}
\]

\hfill$\square$
\end{example}

This Example implies to introduce the following

\begin{definition}\label{sumofgraph}
  Let $\Gamma_1 = (U_1,V_1,E_1)$, $\Gamma_2 = (U_2,V_2,E_2)$ be two bigraphs with injections $\varphi_1:V_1 \to \mathbb{N}$, $\varphi_2: V_2 \to \mathbb{N}$. Their \textit{polynomial sum} $\Gamma_1 \sqcup_{\varphi_1,\varphi_2} \Gamma_2$ is called a graph $\Gamma = (U,V,E)$ with $U = U_1 \sqcup U_2$, $V= V_1 \times V_2/\sim$, where $V_1 \ni v_1\sim v_2 \in V_2$ whenever $\varphi_1(v_1) = \varphi_2(v_2)$, and $E = \cup_{u\in U,v\in V}\{(u,v)\}$.
\end{definition}

\begin{proposition}\label{propsumofgraph}
  Let $\Gamma_1 = (U_1,V_1,E_1)$, $\Gamma_2 = (U_2,V_2,E_2)$ be two bigraphs with injections $\varphi_1:V_1 \to \mathbb{N}$, $\varphi_2: V_2 \to \mathbb{N}$. Then
  \[
   \Gamma_1 \bigsqcup_{\varphi_1,\varphi_2} \Gamma_2 \cong \Gamma(p(\Gamma_1,\varphi_2) + p(\Gamma_2,\varphi_2)).
  \]
\end{proposition}
\begin{proof}
  Indeed, we have
  \[
   p=p(\Gamma_1,\varphi_1) + p(\Gamma_2,\varphi_2) = \sum_{u_1\in U_1}x^{d(u_1)} + \sum_{u_2\in U_2}x^{d(u_2)},
  \]
  where $d(u_i) = \sum_{v_i\in V_i \cap E(u_i)}2^{\varphi_i(v_i)}$, $i =1,2$. Since $V = \tau(p)$: $V = V_1 \times V_2 /\sim$. Further, by Construction \ref{con2}, $U = U_1 \sqcup U_2$, and hence $E =  \cup_{u\in U,v\in V}\{(u,v)\}$, as claimed.
\end{proof}

\begin{corollary}
  If $\mathrm{Im}(\varphi_1) \cap \mathrm{Im}(\varphi_2) = \varnothing$, then $\Gamma_1 \bigsqcup_{\varphi_1,\varphi_2} \Gamma_2 = \Gamma_1 \bigsqcup \Gamma_2$.
\end{corollary}
\begin{proof}
  The proof immediately follows from the previous Proposition.
\end{proof}

\section{Bipartite Digraph Polynomials}
We start from the following procedure allows to construct a polynomial $p(\vec\Gamma) \in \mathbb{N}[x,y]$  from a given bipartite digraph $\vec \Gamma$.

Given a bipartite digraph $\vec \Gamma = (U,V,E)$, set $\mathstrut^\bullet u: = \{v \in V: (v,u)\in E\}$, and $u^\bullet: = \{v \in V: (u,v) \in E\},$ for any $u \in U$.

\begin{construction}\label{construction1}
Let $\vec \Gamma = (U,V,E)$ be a bipartite digraph and $\varphi:V \to \mathbb{N}$ an injection. Set $p(\vec\Gamma,\varphi):=\sum_{u\in U}x^{d_-(u)}y^{d_+(u)}\in \mathbb{N}[x,y],$ where
 \[
 d_-(u): = \sum_{v \in {\mathstrut^\bullet u}}2^{\varphi(v)},  \quad d_+(u): = \sum_{v \in u^\bullet}2^{\varphi(v)},
 \]
 we put $d_-(u) = 0$ (\textit{reps.} $d_+(u) = 0$) if $\mathstrut^\bullet u = \varnothing$ (\textit{resp.} $u^\bullet = \varnothing$).
\end{construction}

\begin{example}
Let us consider the following bipartite digraph $\vec \Gamma = (U,V,E)$ which is shown below
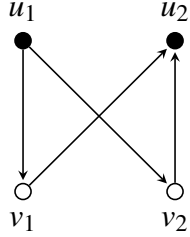
\[
 \begin{tikzpicture}[> = stealth, 
            shorten > = 1pt, 
            auto,
            node distance = 3cm, 
            semithick 
        ]
	\vertex[fill] (u1) at (0,2) [label=above:$u_1$] {};
  \vertex[fill] (u2) at (2,2) [label=above:$u_2$] {};
  \vertex (0) at (0,0) [label=below:$v_1$] {};
	\vertex (1) at (2,0) [label=below:$v_2$] {};
  \path[->]
		    (u1) edge (0)
        (u1) edge (1)
        (0) edge (u2)
        (1) edge (u2);
 \end{tikzpicture}
\]

Define an injections $\varphi:V \to \mathbb{N}$ as follows: $\varphi(v_1) = 0$, $\varphi(v_2) = 1$. We then get
\begin{align*}
  & d_-(u_1) = 0, \quad d_-(u_2) = 2^0 + 2^1 = 3,\\
  & d_+(u_1) = 2^0 + 2^1 = 3, \quad d_+(u_2) = 0,
\end{align*}
hence $p(\vec \Gamma) =   x^3 + y^3.$
\hfill$\square$
\end{example}

\begin{proposition}
Let $\vec\Gamma = (U,V;A)$, $\vec\Gamma' = (U',V';A')$ be two connected bipartite directed graphs; if there is an isomorphism $(\beta_U,\beta_V): \vec \Gamma \to \vec \Gamma'$, then $p(\vec\Gamma,\varphi) = p(\vec \Gamma', \varphi\circ \beta_V^{-1})$, where $\varphi:V \to \mathbb{N}$ is an arbitrary injection.
\end{proposition}
\begin{proof}
Let $p(\vec \Gamma,\varphi) = \sum_{u \in U}x^{d_-(u)}y^{d_+(u)}$ and $\beta_V(v) = v'$, $\beta_U(u) =u'$, for every $v\in V$, $u\in U$. We then obtain
\begin{align*}
  & d_-(u) :=  \sum\limits_{v \in \mathstrut^\bullet u} 2^{\varphi(v)} = \sum\limits_{\beta_V^{-1}(v') \in \mathstrut^\bullet \beta_U^{-1}(u')} 2^{\varphi (\beta_V^{-1} (v'))} = \sum\limits_{v' \in \mathstrut^\bullet u'} 2^{\varphi (\beta^{-1}_V(v'))}:=d_{-}(u')\\
  & d_+(u) :=  \sum\limits_{v \in u^\bullet} 2^{\varphi(v)} = \sum\limits_{\beta_V^{-1}(v') \in \beta_U^{-1}(u')^\bullet} 2^{\varphi (\beta^{-1}_V (v'))} = \sum\limits_{v' \in u'^\bullet} 2^{\varphi (\beta^{-1}_V(v'))}:=d_+(u'),
\end{align*}
and hence $p(\vec\Gamma',\varphi\circ \beta_V^{-1}) = \sum_{u' \in U'}x^{d_-(u')}y^{d_+(u')}=p(\vec \Gamma,\varphi)$, as claimed.
\end{proof}

So, we have constructed a polynomial $p\in\mathbb{N}[x,y]$ for every bipartite directed graph and the previous result implies the following

\begin{corollary}
  Two connected bipartite directed graphs $\vec\Gamma = (U,V;A)$, $\vec \Gamma' = (U',V';A')$ are isomorphic if and only if there are injections $\varphi:V \to \mathbb{N}$, $\varphi':V' \to \mathbb{N}$ such that $p(\vec \Gamma,\varphi) = p(\vec \Gamma',\varphi')$.
\end{corollary}

 Now, we want to present an inverse procedure; for a given polynomial $p(x,y) \in \mathbb{N}[x,y]$ we construct a bipartite directed graph $\vec \Gamma(p) = (U(p),V(p);A(p))$ with a natural injection $\iota_p:V(p) \to \mathbb{N}$ such that $p(\Gamma(p),\iota_p) = p(x,y)$.

\begin{construction}\label{con2}
Let $I\subset \mathbb{N} \times \mathbb{N}$ be a finite subset, $p = p(x,y) = \sum_{(i,j) \in I}n_{ij}x^iy^j \in \mathbb{N}[x,y]$ a polynomial. For every $(i,j) \in I$, we consider a set $U_{ij}:=\{u_1^{(i,j)},\ldots,u_{n_{ij}}^{(i,j)} \}$ of some elements. Set $\vec\Gamma(p): = (U(p),V(p);A(p))$, where
\begin{align*}
  & V(p) :=\tau(P):= \bigcup\nolimits_{(i,j) \in I} \{\tau(i), \tau(j)\},\\
  & U(p) = \bigcup\nolimits_{(i,j)\in I}U_{i,j}, \\
  & A(p) := \bigcup\nolimits_{\substack{(i,j) \in I \\ u^{(i,j)} \in U(p)}}\{(u^{(i,j)}, \tau(j)),(\tau(i), u^{(i,j)})\},
\end{align*}
 and putting $\iota(\tau(i)) := \tau(i)$ for any $\tau(i) \in V(p)$ we thus determine \textit{the natural injection $\iota: V(p) \to \mathbb{N}$.}
\end{construction}

\begin{example}
Let us consider the polynomial $p = p(x,y) = x^5y^3 + 2x^2 + y + 2$. Using Construction \ref{con2}, we obtain
\begin{eqnarray*}
V(p) := \tau(p)&=& \tau(5) \cup \tau(3) \cup \tau(2) \cup \tau(1) \cup \tau(0) \\
 &=&\{2,1\} \cup \{1,0\} \cup \{1\} \cup \{0\} \cup \varnothing = \{2,1,0\},
\end{eqnarray*}
because of $5 = 2^2 + 2^1$, $3= 2^1 + 2^0$ and $2 = 2^1$. Next,
\begin{eqnarray*}
U(p) &=& U_{3,3} \cup U_{2,0} \cup U_{0,1} \cup U_{0,0} \\
 &=& \{u_1^{(3,3)} \} \cup \{u_1^{(2,0)}, u_2^{(2,0)} \} \cup \{u_1^{(0,1)} \} \cup \{u_1^{(0,0)} \},
\end{eqnarray*}
and the corresponding graph can be presented as follows:
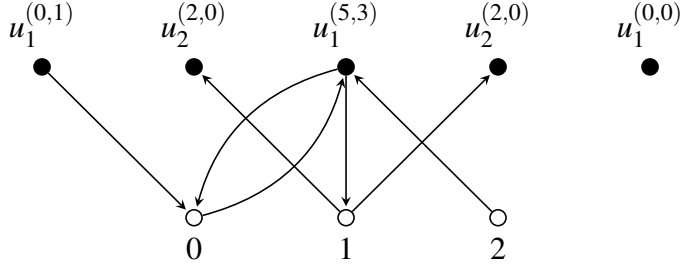
\[
 \begin{tikzpicture}[> = stealth, 
            shorten > = 1pt, 
            auto,
            node distance = 3cm, 
            semithick 
        ]
	\vertex[fill] (u011) at (-2,2) [label=above:$u^{(0,1)}_1$] {};
  \vertex[fill] (u201) at (0,2)  [label=above:$u^{(2,0)}_2$] {};
  \vertex[fill] (u531) at (2,2)  [label=above:$u^{(5,3)}_1$] {};
  \vertex[fill] (u202) at (4,2)  [label=above:$u^{(2,0)}_2$] {};
  \vertex[fill] (u001) at (6,2)  [label=above:$u^{(0,0)}_1$] {};
  \vertex (0) at (0,0) [label=below:$0$] {};
	\vertex (1) at (2,0) [label=below:$1$] {};
  \vertex (2) at (4,0) [label=below:$2$] {};
  \path[->]
		    (u011) edge (0)
        (u531) edge (1)
        (1) edge (u201)
        (1) edge (u202)
        (2) edge (u531)
                ;
\path[->,bend right]
  (u531) edge (0)
  (0) edge (u531)
  ;
 \end{tikzpicture}
\]
\hfill$\square$
\end{example}

\begin{proposition}\label{Prop=}
Let $\vec \Gamma = (U,V;A)$ be a bipartite digraph and $\varphi:V \to \mathbb{N}$ be an injection. The graphs $\vec\Gamma$ and $\vec \Gamma (p(\vec \Gamma,\varphi))$ are isomorphic.
\end{proposition}
\begin{proof}
Let $\vec \Gamma (p(\vec \Gamma, \varphi)) = (\widetilde{U},\widetilde{V}; \widetilde{A})$ and $p(\vec \Gamma,\varphi) = \sum_{u\in U}n_{u}x^{d_-(u)}y^{d_+(u)}$.

From Construction \ref{con2} it follows that $\widetilde{V} = \tau(P(\vec \Gamma,\varphi))= \mathrm{Im}(\varphi)$, since the map $\varphi$ is injective we thus get the following bijection $\beta_V=\varphi^{-1}:\widetilde{V} \cong V:\varphi =\beta^{-1}.$

Set $u^\bullet = \{v_{1,u},\ldots,v_{k,u}\}$ and $\mathstrut^\bullet u = \{v_{u,1},\ldots, v_{u,h}\}$ for any $u\in U$. We then can write
\[
 \widetilde{A}:=A(p(\vec \Gamma,\varphi)) = \bigcup\nolimits_{\substack{u\in U \\ 1 \le i \le k \\ 1 \le j \le h}} \{ (\varphi(v_{i,u}),u^{d_+(u)}) , (u^{d_-(u)}, \varphi(v_{u,j}))\}.
\]

Further, by Construction \ref{construction1}, $\sum_{u\in U}n_{u} = |U|$, and by Construction \ref{con2},$|\widetilde{U}| = \sum_{u \in U}n_{u}$. Hence, there exist bijections between the sets $U$ and $\widetilde{U}$ and this completes the proof.
\end{proof}

\section{Algebraic Operations on Bipartite Digraphs}

\subsection{Polynomial Product of Bipartite Digraphs}
In this subsection we introduce a binary operation on labeled bipartite digraphs. We also (as in 3.1) call these operation a \textit{polynomial product} because of it is defined by the product of its polynomial.

We start with examples.

\begin{example}
Let us consider the following bipartite digraphs: $\vec \Gamma_1 = (V_1,E_1)$, $\vec \Gamma_2 = (V_2,E_2)$ $V_1 = \{v_{11}, v_{12}\}$, $U_1 = \{u_{11}, u_{12}\}$, $V_2 = \{v_{21}, v_{22}\}$, $U_2 = \{u_{21}, u_{22}\}$:
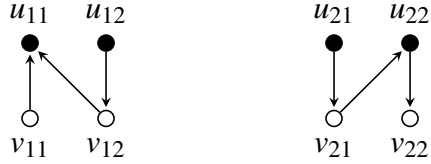
\[
  \begin{tikzpicture}[> = stealth, 
            shorten > = 1pt, 
            auto,
            node distance = 3cm, 
            semithick 
        ]
	\vertex[fill] (u11) at (0,1) [label=above:$u_{11}$] {};
  \vertex[fill] (u12) at (1,1) [label=above:$u_{12}$] {};
  \vertex[fill] (u21) at (4,1) [label=above:$u_{21}$] {};
  \vertex[fill] (u22) at (5,1) [label=above:$u_{22}$] {};
	\vertex (v11) at (0,0) [label=below:$v_{11}$] {};
	\vertex (v12) at (1,0) [label=below:$v_{12}$] {};
  \vertex (v21) at (4,0) [label=below:$v_{21}$] {};
  \vertex (v22) at (5,0) [label=below:$v_{22}$] {};
	\path[->]
		(v11) edge (u11)
		(v12) edge (u11)
    (u12) edge (v12)
    (u21) edge (v21)
    (v21) edge (u22)
    (u22) edge (v22)
		;
\end{tikzpicture}
\]

Let $\varphi_1: V_1 \to \mathbb{N}$, $\varphi_2:V_2 \to \mathbb{N}$ be given as follows:
\[
 \varphi_1(v_{11}) = 0, \varphi_1(v_{12}) = 1, \varphi_2(v_{21}) = 2, \varphi_2(v_{22}) = 3.
\]

We get:
\begin{center}
  \begin{tabular}[t]{ll}
     $d_{-}(u_{11}) = 0,$     & $d_+(u_{11}) = 2^0 + 2^1 = 3,$ \\
     $d_{-}(u_{12}) = 2^1=2,$ & $d_+(u_{12}) = 0,$ \\
     $d_{-}(u_{21}) = 2^2=4,$ & $d_+(u_{21}) = 0,$ \\
     $d_{-}(u_{22}) = 2^3=8,$ & $d_+(u_{22}) = 2^2 = 4$,
  \end{tabular}
\end{center}
then
\begin{align*}
 & p(\vec \Gamma_1,\varphi_1) = x^0y^3 + x^2y^0 = y^3 + x^2, \\
 & p(\vec \Gamma_2,\varphi_2) = x^4y^0 + x^8y^4 = x^4 + x^8y^4.
\end{align*}

Further, let $p = p(\vec \Gamma_1,\varphi_1)\cdot p(\vec \Gamma_2,\varphi_1)$,
\[
 p = x^4y^3 + x^8y^7 + x^6 + x^{10}y^4.
\]

Construct $\vec\Gamma(p)$. Using $10 = 2^3 + 2^1$, $8= 2^3$, $7=2^2 + 2^1 + 1^0$, $6 = 2^2 + 2^1$, $4=2^2$, $3 = 2^1 + 2^0$, and $1 = 2^0$, we get $\tau(p) = \{3,2,1,0\}$.

Next,
\begin{align*}
 & I = \{(10,4), (8,7), (4,3), (6,0)  \}, \\
 & U(p) = \{u^{(10,4)}_1, u^{(8,7)}_1, u^{(4,3)}_1, u^{(6,0)}_1 \}.
\end{align*}

We thus get the following graph $\vec \Gamma(p)$
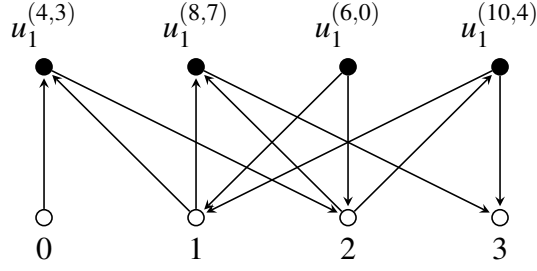
\[
 \begin{tikzpicture}[> = stealth, 
            shorten > = 1pt, 
            auto,
            node distance = 3cm, 
            semithick 
        ]
	\vertex[fill] (u431) at (-2,2) [label=above:$u^{(4,3)}_1$] {};
  \vertex[fill] (u871) at (0,2)  [label=above:$u^{(8,7)}_1$] {};
  \vertex[fill] (u601) at (2,2)  [label=above:$u^{(6,0)}_1$] {};
  \vertex[fill] (u1041) at (4,2)  [label=above:$u^{(10,4)}_1$] {};
  \vertex (0) at (-2,0)[label=below:$0$] {};
	\vertex (1) at (0,0) [label=below:$1$] {};
  \vertex (2) at (2,0) [label=below:$2$] {};
  \vertex (3) at (4,0) [label=below:$3$] {};
  \path[->]
		    (u431) edge (2)
        (0) edge (u431)
        (1) edge (u431)
        (u871) edge (3)
        (1) edge (u871)
        (2) edge (u871)
        (u601) edge (1)
        (u601) edge (2)
        (u1041) edge (1)
        (u1041) edge (3)
        (2) edge (u1041)
                ;
 \end{tikzpicture}
\]

It is not so hard to see that $\vec \Gamma(p) \cong \Gamma_1 \times \Gamma_2$. Indeed, set
\begin{center}
  \begin{tabular}[t]{llll}
    $0 \leftrightarrow v_{11},$ & $1 \leftrightarrow v_{12},$ & $2 \leftrightarrow v_{21},$ & $3 \leftrightarrow v_{22},$ \\
    $u_1^{(4,3)} \leftrightarrow (u_{11}, u_{21}),$ & $u_1^{(8,7)} \leftrightarrow (u_{11}, u_{22}),$ & $u_1^{(6,0)} \leftrightarrow (u_{12}, u_{21}),$ & $u_1^{(10,4)} \leftrightarrow (u_{12}, u_{22}),$
  \end{tabular}
\end{center}
and the isomorphism is clear.

\hfill$\square$
\end{example}

\begin{example}
  Let us consider the same bipartite digraphs as in the previous example: $\vec \Gamma_1 = (V_1,E_1)$, $\vec \Gamma_2 = (V_2,E_2)$ $V_1 = \{v_{11}, v_{12}\}$, $U_1 = \{u_{11}, u_{12}\}$, $V_2 = \{v_{21}, v_{22}\}$, $U_2 = \{u_{21}, u_{22}\}$:
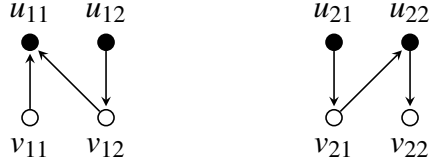
\[
  \begin{tikzpicture}[> = stealth, 
            shorten > = 1pt, 
            auto,
            node distance = 3cm, 
            semithick 
        ]
	\vertex[fill] (u11) at (0,1) [label=above:$u_{11}$] {};
  \vertex[fill] (u12) at (1,1) [label=above:$u_{12}$] {};
  \vertex[fill] (u21) at (4,1) [label=above:$u_{21}$] {};
  \vertex[fill] (u22) at (5,1) [label=above:$u_{22}$] {};
	\vertex (v11) at (0,0) [label=below:$v_{11}$] {};
	\vertex (v12) at (1,0) [label=below:$v_{12}$] {};
  \vertex (v21) at (4,0) [label=below:$v_{21}$] {};
  \vertex (v22) at (5,0) [label=below:$v_{22}$] {};
	\path[->]
		(v11) edge (u11)
		(v12) edge (u11)
    (u12) edge (v12)
    (u21) edge (v21)
    (v21) edge (u22)
    (u22) edge (v22)
		;
\end{tikzpicture}
\]

Let $\varphi_1: V_1 \to \mathbb{N}$, $\varphi_2:V_2 \to \mathbb{N}$ be given as follows:
\[
 \varphi_1(v_{11}) = 0, \varphi_1(v_{12}) = 1, \varphi_2(v_{21}) = 1, \varphi_2(v_{22}) = 2.
\]

We get:
\begin{center}
  \begin{tabular}[t]{ll}
     $d_{-}(u_{11}) = 0,$     & $d_+(u_{11}) = 2^0 + 2^1 = 3,$ \\
     $d_{-}(u_{12}) = 2^1=2,$ & $d_+(u_{12}) = 0,$ \\
     $d_{-}(u_{21}) = 2^1=2,$ & $d_+(u_{21}) = 0,$ \\
     $d_{-}(u_{22}) = 2^2=4,$ & $d_+(u_{22}) = 2^1 = 2$,
  \end{tabular}
\end{center}
then
\begin{align*}
 & p(\vec \Gamma_1,\varphi_1) = x^0y^3 + x^2y^0 = y^3 + x^2, \\
 & p(\vec \Gamma_2,\varphi_2) = x^2y^0 + x^4y^2 = x^2 + x^4y^2.
\end{align*}

Further, let $p = p(\vec \Gamma_1,\varphi_1)p(\vec \Gamma_2,\varphi_1)$,
\[
 p = x^2y^3 + x^4y^5 + x^4 + x^6y^2.
\]

We obtain
\begin{align*}
 & \tau(p) = \{ 2,1,0 \},\\
 & I = \{ (6,2), (4,5), (2,3), (4,0) \}, \\
 & U(p) = \{u_1^{(6,2)}, u_1^{(4,5)}, u_1^{(2,3)}, u_1^{(4,0)}\}.
\end{align*}

Hence, the graph $\vec \Gamma (p)$ has the following form:
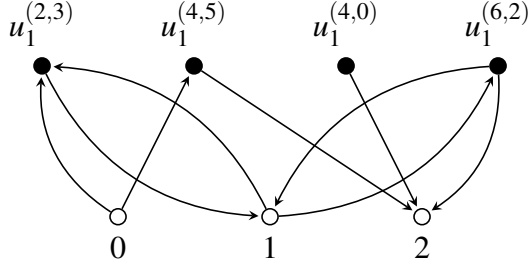
\[
 \begin{tikzpicture}[> = stealth, 
            shorten > = 1pt, 
            auto,
            node distance = 3cm, 
            semithick 
        ]
	\vertex[fill] (u231) at (-2,2) [label=above:$u^{(2,3)}_1$] {};
  \vertex[fill] (u451) at (0,2)  [label=above:$u^{(4,5)}_1$] {};
  \vertex[fill] (u401) at (2,2)  [label=above:$u^{(4,0)}_1$] {};
  \vertex[fill] (u621) at (4,2)  [label=above:$u^{(6,2)}_1$] {};
  \vertex (0) at (-1,0)[label=below:$0$] {};
	\vertex (1) at (1,0) [label=below:$1$] {};
  \vertex (2) at (3,0) [label=below:$2$] {};
  \path[->]
		    (0) edge (u451)
        (u451) edge (2)
        (u401) edge (2)
                     ;
  \path[->,bend right]
  (u231) edge (1)
  (1) edge (u231)
  (1) edge (u621)
  (u621) edge (1)
    ;
  \path[->, bend left]
  (0) edge (u231)
  (u621) edge (2)
  ;
 \end{tikzpicture}
\]

\hfill $\square$
\end{example}

As in the case of undirected bigrpaphs, we also see that the graph $\vec\Gamma(p)$ for the previous examples looks like the product $\vec\Gamma_1 \times \vec\Gamma_2$ ``perturbed'' by elements of the set $\mathrm{Im}(\varphi_1) \cap \mathrm{Im}(\varphi_2)$. This example motivates to introduce the following construction.

\begin{definition}
  Let $\vec \Gamma_1 = (U_1,V_1,A_1)$, $\vec \Gamma_2 = (U_2,V_2,A_2)$ be two directed bigraphs, $\varphi_1: V_1 \to \mathbb{N}$, $\varphi_2:V_2 \to \mathbb{N}$ two arbitrary injections. Their \textit{polynomial product} $\vec\Gamma_1 \times_{\varphi_1,\varphi_2}\vec \Gamma_2$, is defined as follows
  \[
   \vec\Gamma_1 \times_{\varphi_1,\varphi_2}\vec \Gamma_2:=\vec\Gamma(p(\vec \Gamma_1,\varphi_1) \cdot p(\Gamma_2,\varphi_2)).
  \]
\end{definition}

On the other hand, using Construction \ref{construction1}, definition of product of polynomials in $\mathbb{N}[x,y]$, and Construction \ref{con2}, we get

\begin{proposition}
  The polynomial product $\vec\Gamma_1 \times_{\varphi_1,\varphi_2}\vec \Gamma_2$ can be described as a directed bigraph $\Gamma = (V,U,A)$ with
  \begin{align*}
    & V = \bigcup_{u_1 \in U_1,u_2 \in U_2} \{\tau(d_{-}(u_1) + d_{-}(u_2))\} \cup \{ \tau ( d_{+}(u_1) + d_{+}(u_2) ) \},\\
    & U = U_1 \times U_2,\\
    & A = \bigcup_{u_1 \in U_1, u_2 \in U_2} \bigcup\nolimits_{\substack{ n \in \tau(d_{-}(u_1) + d_{-}(u_2)) \\ m \in \tau(d_{+}(u_1) + d_{+}(u_2)) }} \{ \bigl((u_1,u_2), n\bigr), \bigl(m, (u_1,u_2)\bigr)\}.
  \end{align*}
\end{proposition}
\begin{proof}
  The proof is straightforward.
\end{proof}

\begin{corollary}
  If $\mathrm{Im}(\varphi_1) \cap \mathrm{Im}(\varphi_2) = \varnothing$ then $\vec \Gamma_1 \times_{\varphi_1,\varphi_2}\vec\Gamma_2 = \vec \Gamma_1\times \vec \Gamma_2.$
\end{corollary}
\begin{proof}
  The proof immediately follows by using Lemma \ref{tau}.
\end{proof}

\subsection{Polynomial Coproduct of Directed Bigraphs} Here we show that sum of two directed bipartite graph polynomials correspondences to a simple operations on labeled directed bigraphs.

We start with an Example.

\begin{example}
  Let us turn to the previous example. We have: $\vec \Gamma_1 = (V_1,E_1)$, $\vec \Gamma_2 = (V_2,E_2)$ $V_1 = \{v_{11}, v_{12}\}$, $U_1 = \{u_{11}, u_{12}\}$, $V_2 = \{v_{21}, v_{22}\}$, $U_2 = \{u_{21}, u_{22}\}$:
\[
  \begin{tikzpicture}[> = stealth, 
            shorten > = 1pt, 
            auto,
            node distance = 3cm, 
            semithick 
        ]
	\vertex[fill] (u11) at (0,1) [label=above:$u_{11}$] {};
  \vertex[fill] (u12) at (1,1) [label=above:$u_{12}$] {};
  \vertex[fill] (u21) at (4,1) [label=above:$u_{21}$] {};
  \vertex[fill] (u22) at (5,1) [label=above:$u_{22}$] {};
	\vertex (v11) at (0,0) [label=below:$v_{11}$] {};
	\vertex (v12) at (1,0) [label=below:$v_{12}$] {};
  \vertex (v21) at (4,0) [label=below:$v_{21}$] {};
  \vertex (v22) at (5,0) [label=below:$v_{22}$] {};
	\path[->]
		(v11) edge (u11)
		(v12) edge (u11)
    (u12) edge (v12)
    (u21) edge (v21)
    (v21) edge (u22)
    (u22) edge (v22)
		;
\end{tikzpicture}
\]
and $\varphi_1: V_1 \to \mathbb{N}$, $\varphi_2:V_2 \to \mathbb{N}$ are given as follows:
\[
 \varphi_1(v_{11}) = 0, \varphi_1(v_{12}) = 1, \varphi_2(v_{21}) = 1, \varphi_2(v_{22}) = 2,
\]
and then
\begin{align*}
 & p(\vec \Gamma_1,\varphi_1) = x^0y^3 + x^2y^0 = y^3 + x^2, \\
 & p(\vec \Gamma_2,\varphi_2) = x^2y^0 + x^4y^2 = x^2 + x^4y^2.
\end{align*}

Consider the following polynomial $p:=p(\vec\Gamma_1,\varphi_1) + p(\vec \Gamma_2,\varphi_2)$. We obtain:
\begin{eqnarray*}
  p &=& x^4y^2 + 2x^2 + y^3,\\
  \tau(p) &=& \{\tau(4),\tau(3),\tau(2)\} = \{2,1,0\},\\
  I(p) &=& \{(4,2), (2,0), (0,3)\}, \\
  U(p) &=& \{u_1^{(4,2)}, u_1^{(2,0)}, u_2^{(2,0)}, u_1^{(0,3)}\},
\end{eqnarray*}
and the graph $\vec \Gamma(p)$ has the following form
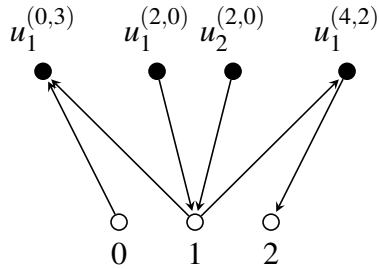
\[
 \begin{tikzpicture}[> = stealth, 
            shorten > = 1pt, 
            auto,
            node distance = 3cm, 
            semithick 
        ]
	\vertex[fill] (u031) at (-1,2) [label=above:$u_1^{(0,3)}$] {};
  \vertex[fill] (u201) at (0.5,2) [label=above:$u_1^{(2,0)}$] {};
  \vertex[fill] (u202) at (1.5,2) [label=above:$u_2^{(2,0)}$] {};
  \vertex[fill] (u421) at (3,2) [label=above:$u_1^{(4,2)}$] {};
	\vertex (0) at (0,0) [label=below:$0$] {};
	\vertex (1) at (1,0) [label=below:$1$] {};
  \vertex (2) at (2,0) [label=below:$2$] {};
	\path[->]
		    (0) edge (u031)
        (1) edge (u031)
        (u201) edge (1)
        (u202) edge (1)
        (1) edge (u421)
        (u421) edge (2)
    ;
\end{tikzpicture}
\]

\hfill $\square$
\end{example}

Roughly speaking, the graph $\vec\Gamma(p)$ is an attaching the graph $\vec\Gamma_1$ to $\vec \Gamma_2$ by the map $\varphi:=\varphi_1 \sqcup \varphi_2$.

\begin{definition}
  Let $\vec\Gamma_1 = (U_1,V_1,A_1)$, $\vec\Gamma_2 = (U_2,V_2,A_2)$ be two directed bigraphs with injections $\varphi_1:V_1 \to \mathbb{N}$, $\varphi_2: V_2 \to \mathbb{N}$. Their \textit{polynomial sum} $\vec\Gamma_1 \sqcup_{\varphi_1,\varphi_2}\vec \Gamma_2$ is defined as follows
  \[
   \vec\Gamma_1 \sqcup_{\varphi_1,\varphi_2}\vec \Gamma_2:=\vec\Gamma(p(\vec\Gamma_1,\varphi_1) + p(\vec \Gamma_2,\varphi_2)).
  \]
\end{definition}

On the other hand, by Constructions \ref{construction1}, \ref{con2}, we get the following description of this operation.

\begin{proposition}
  $\vec\Gamma = (U,V,A)$, where
  \begin{align*}
    & U = U_1 \sqcup U_2,\\
    & V= V_1 \times V_2/\sim ,\\
    &A  = \cup_{u\in U,v\in V}\{(u,v)\},
  \end{align*}
  here $\sim$ is defined as follows: $v_1 \sim v_2$ whenever  $\varphi_1(v_1) = \varphi_2(v_2)$, for $v_1 \in V_1$, $v_2\in V_2$.
\end{proposition}
\begin{proof}
  The proof is straightforward.
\end{proof}

\begin{corollary}
  If $\mathrm{Im}(\varphi_1) \cap \mathrm{Im}(\varphi_2) = \varnothing$, then $\vec\Gamma_1 \bigsqcup_{\varphi_1,\varphi_2} \vec\Gamma_2 = \vec\Gamma_1 \bigsqcup \vec\Gamma_2$.
\end{corollary}
\begin{proof}
  The proof immediately follows from Lemma \ref{tau}.
\end{proof}

\section{Dividing of polynomials with coefficients in natural numbers}

In this section we discuss a criteria of dividing polynomials semirings $\mathbb{N}[x]$, $\mathbb{N}[x,y]$ via bipartite graphs.

We consider only undirected bipartite graphs (\textit{i.e.,} we consider dividing in the semiring $\mathbb{N}[x]$), but analogous results (for $\mathbb{N}[x,y]$), using Section 4,5, one can easy obtain for directed bipartite graphs.

As we have already known any (directed) bipartite graph can be considered as a polynomial $p\in \mathbb{N}[x]$ (\textit{resp.} $\in \mathbb{N}[x,y]$) and vise verse. Thus, if we take a polynomial, say, $p\in \mathbb{N}[x]$ then $p$ is irreducible if and only if the graph $\Gamma(p)$ cannot be presented as a polynomial product $\Gamma_1\times_{\varphi_1,\varphi_2} \Gamma_2$, of two graphs $\Gamma_1 = (U_1,V_1,E_1)$, $\Gamma_2=(U_2,V_2,E_2)$ with $E_1,E_2 \ne \varnothing$ and where $\varphi_1:V_1 \to \mathbb{N}$, $\varphi_2: V_2 \to \mathbb{N}$ are suitable injections.

To be more precisely let us introduce the following concept.

\begin{definition}
  A (directed) bipartite graph $\Gamma = (U,V,E)$ is called \textit{irreducible} if there is no an injection $\varphi:V \to \mathbb{N}$, and there are no two graphs $\Gamma_1 = (U_1,V_1,E_1)$, $\Gamma_2 = (U_2,V_2,E_2)$ with $E_1,E_2 \ne \varnothing$, such that $\Gamma = \Gamma_1 \times_{\varphi_1,\varphi_2} \Gamma_2$, where $\varphi_1:=\varphi |_{V_1}$, $\varphi_2:=\varphi|_{V_2}$.

  Otherwise, a graph $\Gamma$ is called \textit{reducible,} and in the case $\Gamma = \Gamma_1 \times_{\varphi_1,\varphi_2} \Gamma_2$, we say that $\Gamma_1$, $\Gamma_2$ are its \textit{polynomial factors.}
\end{definition}

Thus, we get the following criteria of dividing polynomials in $\mathbb{N}$.
\begin{theorem}\label{criteriafordiv}
  A polynomial $p\in\mathbb{N}[x]$ is irreducible if and only if the graph $\Gamma(p)$ is irreducible.

  Moreover, if the graph $\Gamma(p)$ is not irreducible, say, $\Gamma = \Gamma_1 \times_{\iota_1,\iota_2} \Gamma_2$, where $\Gamma_1 = (U_1,V_1,E_1)$, $\Gamma_1 = (U_2,V_2,E_2)$, and $\iota_1:=\iota |_{V_1}$, $\iota_2:=\iota|_{V_2}$, then $p= p(\Gamma_1,\iota_1)\cdot p(\Gamma_2,\iota_2)$ in $\mathbb{N}[x]$.
\end{theorem}
\begin{proof}
  The proof is immediately follows from Construction \ref{conpolgraph}, Construction \ref{congraphpol}, and Proposition \ref{product}.
\end{proof}

Let us demonstrate this idea in the following Example.

\begin{example}\label{ex}
  Let us consider a polynomial $p = x^3 + 2x^2 + 2x + 1$. Using Construction \ref{conpolgraph}, we have $3 = 2^1 +2^0$, $2 = 2^1$, $1 = 2^0$, hence $\tau(p) = \{0,1\}$. Further, $I(p) = \{3,2,1,0\}$, and then $U(p) = \{u_1^3, u_1^2,u_2^2, u_1^1,u_2^1, u_1^\varnothing\}$, and therefore we get the following graph $\Gamma(p):$

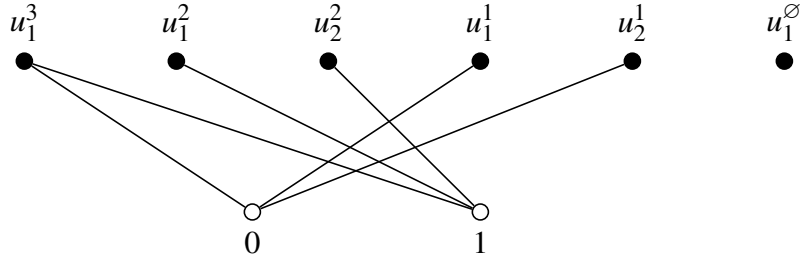
\[
 \begin{tikzpicture}[semithick]
  \vertex[fill] (u31) at (-3,2) [label=above:$u_1^3$] {};
	\vertex[fill] (u21) at (-1,2) [label=above:$u_1^2$] {};
  \vertex[fill] (u22) at (1,2) [label=above:$u^2_2$] {};
  \vertex[fill] (u11) at (3,2) [label=above:$u^1_1$] {};
  \vertex[fill] (u12) at (5,2) [label=above:$u^1_2$] {};
  \vertex[fill] (u01) at (7,2) [label=above:$u^\varnothing_1$] {};
	\vertex (0) at (0,0) [label=below:$0$] {};
	\vertex (1) at (3,0) [label=below:$1$] {};
	\path
		    (u31) edge (0)
        (u31) edge (1)
        (u21) edge (1)
        (u22) edge (1)
        (u11) edge (0)
        (u12) edge (0)
        ;
\end{tikzpicture}
\]

Let us consider now the following graphs $\Gamma_1 = (\{a,b\},\{v_{11}\},E_1),$ and $\Gamma_2= (\{c,d,e\},\{v_{21},v_{22}\},E_2)$:
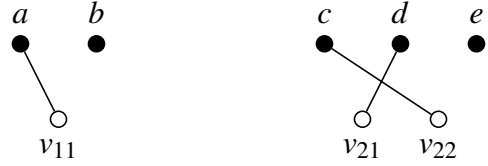
\[
 \begin{tikzpicture}[semithick]
	\vertex[fill] (u11) at (0,1) [label=above:$a$] {};
  \vertex[fill] (u12) at (1,1) [label=above:$b$] {};
  \vertex[fill] (u21) at (4,1) [label=above:$c$] {};
  \vertex[fill] (u22) at (5,1) [label=above:$d$] {};
  \vertex[fill] (u23) at (6,1) [label=above:$e$] {};
	\vertex (v11) at (0.5,0) [label=below:$v_{11}$] {};
	\vertex (v21) at (4.5,0) [label=below:$v_{21}$] {};
  \vertex (v22) at (5.5,0) [label=below:$v_{22}$] {};
	\path
		(u11) edge (v11)
		(u21) edge (v22)
    (u22) edge (v21)
    		;
\end{tikzpicture}
\]
and set
$\varphi_1(v_{11}) = \varphi_2(v_{21}) = 0$, $\varphi_2(v_{22}) = 1$. One can easy verify that $\Gamma(p) = \Gamma_1\times_{\varphi_1,\varphi_2}\Gamma_2$. Indeed, by
\begin{align*}
 &u_1^3 \leftrightarrow (a,c), \qquad u^2_1 \leftrightarrow (a,d), \\
 &u_2^2 \leftrightarrow (b,c), \qquad u^1_1 \leftrightarrow (a,e), \\
 &u_2^1 \leftrightarrow (b,d), \qquad u^\varnothing_1 \leftrightarrow (b,e),
\end{align*}
the isomorphism is clear.

Further, by Construction \ref{congraphpol}, $p(\Gamma_1,\varphi_1) = x+1$, $p(\Gamma_2,\varphi_2) = x^2+x+1$. Hence
\[
 p=x^3 + 2x^2 + 2x + 1 = (x+1)(x^2+x+1).
\]

\hfill$\square$
\end{example}

\begin{remark}
  We have seen that this point of view on dividing in semerings $\mathbb{N}[x]$, $\mathbb{N}[x,y]$ looks interesting, and the authors are going to study these problems (to simplify the criteria) in the future papers.

  \hfill$\square$
\end{remark}

\section{The Zariski Topology on Bipartite Graphs}

As well known, every commutative (semi)ring can be endowed with the Zariski topology \cite{G}.

Namely, as in the case of (associative commutative) rings, we can introduce ideals of semirings. An {\it ideal} $I$ of a semiring $R$ is a nonempty subset of $R$ satisfying the following conditions: (1) if $a,b \in I$ then $a+b \in I$; (2) if $a \in I$ and $r\in R$ then $ra\in I$; (3) $I \ne R$.

Let $R$ be an a (commutative) semiring with unit, an ideal $\mathfrak{p} \subset R$ is called prime if and only if whenever $a\cdot b \in \mathfrak{p}$, for $a,b \in R$, we must have either $a \in \mathfrak{p}$ or $b \in \mathfrak{p}.$ The {\it spectrum} of $R$, denoted $\mathrm{Spec}(R)$, is the set of all prime ideals of $R$. The set $\mathrm{Spec}(R)$ can be equipped with the Zariski topology, for which the closed sets are the sets
\[
 Z(I):=\bigl\{\mathfrak{p}\in \mathrm{Spec}(R)\, | \,I\subseteq \mathfrak{p}\bigr\}
\]
where $I$ is an ideal. Then, it is easy to see that:
\begin{itemize}
  \item[(1)] $Z(\sum_{\alpha \in A}I_\alpha ) = \bigcap_{\alpha \in A}Z(I_\alpha)$, for every family $\{I_\alpha\}_{\alpha \in A}$ of ideals of $R$,
  \item[(2)] $Z(I) \cup Z(J) = Z(IJ) = Z(I \cap J)$ for every pair $I,J$ of ideals of $R$.
\end{itemize}

A basis for the Zariski topology can be constructed as follows. For $f\in R$, define $D_f:=\mathrm{Spec}(R)\setminus Z((f))$. Then each $D_f$ is an open subset of $\mathrm{Spec}(R)$, and $\{D_f|f\in R\}$ is a basis for the Zariski topology.

We are now able to present the following

\begin{theorem}\label{general}
  Let $\Gamma$ be a finite (directed) bipartite graph. Let $Z(\Gamma)$ be a set $\{\Gamma'\}$ of all reducible finite (directed) bipartite graphs, such that $\Gamma'$ is a polynomial factor of $\Gamma$. The set $\mathscr{P}_{\Phi} := \bigl\{(\Gamma = (U,V,E),\varphi)\bigr\}$ of all finite (directed) bipartite graphs with fixed injections $\varphi:V \to \mathbb{N}$, can be endowed with a topology (the Zariski topology) which is a collection of the closed subsets $Z(N)$.
\end{theorem}
\begin{proof}
  It is obviously that this topology arises from the Zariski topology on $\mathrm{Spec}\,\mathbb{N}[x]$ in the case undirected bipartite graphs, and on $\mathrm{Spec}\, \mathbb{N}[x,y]$ in the case directed bipartite graphs. Using Constructions \ref{congraphpol}, \ref{conpolgraph} (\textit{resp.} Construction\ref{construction1}, \ref{con2}), we complete the proof.
\end{proof}

\begin{corollary}
  Every irreducible (directed) bipartite graph $\Gamma$ is the closed point, namely $\Gamma$, in the topological space $\mathscr{P}_\Phi$.
\end{corollary}
\begin{proof}
  As well known, closed points in the Zariski topology of $\mathbb{N}[x]$ correspondence to maximal ideals of $\mathbb{N}[x]$. Since $p$ is undecomposable then from Theorem \ref{criteriafordiv} it follows that the ideal $(p(\Gamma,\varphi))$ is maximal, for an arbitrary injection $\varphi$. This completes the proof.
\end{proof}

\section{Application: Paralyzation of Petri Nets}
Petri nets are a tool for graphical and mathematical simulation, applicable to many systems. The are systems for describing and studying information processing systems that are characterized as being concurrent, asynchronous, distributed, parallel, nondeterministic, and/or stochastic. As a graphical tool, Petri nets can be used as a visual communication aid similar to flow charts, block diagrams, and networks. In addition, tokens are used in these nets to simulate the dynamics and concurrent activities of systems. As far as its being a mathematical tool, it is possible to set up state equations, algebraic equations, and other mathematical models governing the behavior of systems.

G. Winskel in \cite{W87} noticed that Petri nets can be viewed as certain 2-sorted algebras; it allows defining the concept of morphisms for Petri nets as homomorphisms of the corresponding algebras. Thus, the category $\mathbf{PN}$ of Petri nets is defined. The product of two Petri nets is defined in \cite{Win}.

In this section we recall some basic definitions of Petri nets theory. We essentially follow \cite {W87, Win} to define morphisms and product for Petri nets and we thus define Petri net category $\mathbf{PN}.$

\begin{definition}
A Petri net $N$ is a quadruple $(B,E, \mathsf{pre},\mathsf{post})$, where
\begin{itemize}
     \item[(1)] $B$, $E$ are disjoint finite sets of {\it conditions} and {\it events}, respectively,
     \item[(2)] $\mathsf{pre}:E \to 2^B$ is the {\it precondition} map such that $\mathsf{pre}(e)$ is nonempty for all $e \in E$,
     \item[(3)] $\mathsf{post}:E \to 2^B$ is the {\it postcondition} map such that $\mathsf{post}(e)$ is nonempty for all $e \in E$.
\end{itemize}
\end{definition}

Petri nets have a well-known graphical representation in which events are represented as boxes and conditions as circles with directed arcs between them (see fig.\ref{F1}).

\begin{figure}[h!]
\begin{tikzpicture}[node distance=1.3cm,>=stealth',auto]
  \tikzstyle{place}=[circle,thick,draw=blue!75,fill=blue!20,minimum size=6mm]
   \tikzstyle{transition}=[rectangle,thick,draw=black!75,
  			  fill=black!20,minimum size=4mm]
      \node [place, label=above:$b_0$]
                      (w1')                                                {};
    \node [place]     (c1') [below of=w1', label=below:$b_1$]              {};
    \node [place] (s1') [below of=c1',xshift=-7mm, label = left:$b_2$]      {};
    \node [place]
                      (s2') [below of=c1',xshift=7mm, label=right:$b_3$] {};
    \node [place]     (c2') [below of=s1',xshift=7mm, label=above:$b_4$]          {};
    \node [place]
                      (w2') [below of=c2', label=below:$b_5$]          {};
    \node [transition] (e1') [left of=c1'] {$e_1$}
      edge [pre,bend left]                  (w1')
      edge [post]                           (s1')
      edge [pre]                            (s2')
      edge [post]                           (c1');
    \node [transition] (e2') [left of=c2'] {$e_3$}
      edge [pre,bend right]                 (w2')
      edge [post]                           (s1')
      edge [pre]                            (s2')
      edge [post]                           (c2');
    \node [transition] (l1') [right of=c1'] {$e_2$}
      edge [pre]                            (c1')
      edge [pre]                            (s1')
      edge [post]                           (s2')
      edge [post,bend right] node[swap] {} (w1');
    \node [transition] (l2') [right of=c2'] {$e_4$}
      edge [pre]                            (c2')
      edge [pre]                            (s1')
      edge [post]                           (s2')
      edge [post,bend left]  node {}       (w2');
\end{tikzpicture}
\caption{The Petri net $N = (B,E,\mathsf{pre},\mathsf{post})$ is shown, where $B = \{b_0,b_1,b_2,b_3,b_4,b_5\},$ $E = \{e_1,e_2,e_3,e_4\},$ $\mathsf{pre}(e_1) = \{b_0,b_3\}$, $\mathsf{post}(e_1) = \{b_1,b_2\}$, \textit{etc.} }\label{F1}
\end{figure}
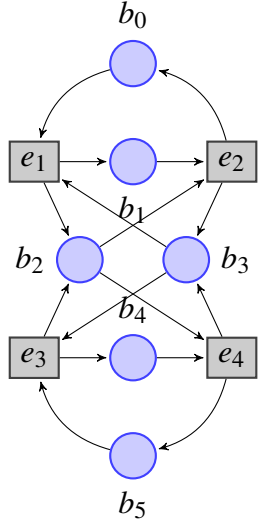

Let $N =(B,E, \mathsf{pre},\mathsf{post})$ be a Petri net with the events $E$. Define $E_*: = E \cup \{*\}$. We extend the $\mathsf{pre}$ and $\mathsf{post}$ condition maps to $*$ by taking $\mathsf{pre}(*) = \varnothing$, $\mathsf{post}(*) = \varnothing$. We also write $\mathstrut^\bullet e$ for the preconditions, $\mathsf{pre}(e)$ and $e^\bullet$ for the postcondition, $\mathsf{post}(e)$, of $e \in E_*$. We write $\mathstrut^\bullet e^\bullet$ for $\mathstrut^\bullet e  \cup e^\bullet$.

\begin{definition}[{\textit{cf.}~\cite{W87, Win}}]
Let $N =(B,E, \mathsf{pre},\mathsf{post})$, $N' =(B',E', \mathsf{pre}',\mathsf{post}')$ be Petri nets. A morphism $(\beta,\eta):N \to N'$ consists of a relation $\beta \subseteq B \times B'$, such that $\beta^{\mathrm{op}}$ (=opposite to $\beta$) is a partial function $B' \to B$, and a partial function $\eta:E \to E'$ such that $\beta (\mathstrut^\bullet e) = \mathstrut^\bullet \eta(e)$, and $\beta(e^\bullet) = \eta(e)^\bullet.$ On the other hand, the diagrams
\[
 \xymatrix{
  E_* \ar@{->}[r]^{\mathsf{pre}} \ar@{->}[d]_{\eta_*} & 2^B \ar@{->}[d]^{2^\beta}\\
  E_*' \ar@{->}[r]_{\mathsf{pre}'}  & 2^{B'}
 }
 \qquad
 \xymatrix{
  E_* \ar@{->}[r]^{\mathsf{post}} \ar@{->}[d]_{\eta_*} & 2^B \ar@{->}[d]^{2^\beta}\\
  E_*' \ar@{->}[r]_{\mathsf{post}'}  & 2^{B'}
 }
\]
are commutative.
\end{definition}

\begin{proposition}[{\cite[Proposition 44]{Win}}]
Nets and their morphisms form a category $\mathbf{PN}$, in which the composition of two morphisms $(\beta,\eta):N\to N'$ and $(\beta',\eta'):N' \to N''$ is $(\beta\circ \beta', \eta \circ \eta'):N \to N''$ (composition in the left component being that of relations and in the right that of partial functions).
\end{proposition}

\begin{remark}[{\bf Isomorphism in the category $\mathbf{PN}$}]\label{Isom}
Recall that a morphism $f:X \to Y$ in a category $\mathscr{C}$ is an isomorphism if it admits a two-sided inverse, meaning that there is another morphism $f^{-1}:Y\to X$ such that $f^{-1}\circ f = \mathbf{id}_X$ and $f\circ f^{-1} = \mathbf{id}_{Y}$, where $\mathbf{id}_X$ and $\mathbf{id}_Y$ are the identity morphisms of $X$ and $Y$, respectively. We use the standard notation $f:X \cong Y$.

Thus, the morphism $(\beta,\eta):N \to N'$ is an isomorphism in the category $\mathbf{PN}$, if there is another morphism $(\beta^{-1},\eta^{-1})$ such that the following diagrams
\[
 \xymatrix{
  & E_* \ar@{->}[rr]^{\mathsf{pre}}  \ar@{->}[ld]_{\eta_*}  \ar@{->}[dd]^(.3){\mathbf{id}_{E_*}} && 2^B  \ar@{->}[ld]_{2^\beta}  \ar@{->}[dd]_{2^{\mathbf{id}_B}}\\
  E'_*  \ar@{->}[rr]_(0.7){\mathsf{pre}'}  \ar@{->}[rd]_{\eta^{-1}_*} && 2^{B'}  \ar@{->}[rd]_(0.4){2^{\beta^{-1}}}& \\
  & E_*  \ar@{->}[rr]_{\mathsf{pre}} && 2^B
 }
 \qquad
  \xymatrix{
  & E_* \ar@{->}[rr]^{\mathsf{post}}  \ar@{->}[ld]_{\eta_*}  \ar@{->}[dd]^(.3){\mathbf{id}_{E_*}} && 2^B  \ar@{->}[ld]_{2^{\beta^{-1}}}  \ar@{->}[dd]_{2^{\mathbf{id}_B}}\\
  E'_*  \ar@{->}[rr]_(0.7){\mathsf{post}'}  \ar@{->}[rd]_{\eta_*^{-1}} && 2^{B'}  \ar@{->}[rd]_(0.4){2^{\beta^{-1}}}& \\
  & E_*  \ar@{->}[rr]_{\mathsf{post}} && 2^B
 }
\]
are commutative. It follows that, in the case when $N$ and $N'$ are finite then an isomorphism between them is a pair $(\beta,\eta)$ of two bijections such that the aforementioned diagrams are commutative.
\end{remark}

In \cite{Win}, it was defined the product of Petri nets.

\begin{figure}[h!]
\begin{center}
\begin{tikzpicture}[node distance=1.3cm,>=stealth',bend angle=25,auto]
  \tikzstyle{place}=[circle,thick,draw=blue!75,fill=blue!20,minimum size=6mm]
  \tikzstyle{red place}=[place,draw=red!75,fill=red!20]
  \tikzstyle{transition}=[rectangle,thick,draw=black!75,
  			  fill=black!20,minimum size=6mm]
  \begin{scope}[yshift=-0.5cm]
   \node [place, label=left:$b_{11}$]         (v1)                                 {};
   \node [transition] (ee1) [above of=v1] {$e$}
     edge [pre, bend left]   (v1)
     edge [post, bend right] (v1);
    \end{scope}
  \begin{scope}[xshift=1.25cm]
    \node(0,0) {$\times$};
  \end{scope}
\begin{scope}[xshift=2.5cm]
 \node [place, label = right:$b_{21}$] (v1')                 {};
 \node [transition] (b) [above of = v1'] {$e_{1}$}
   edge [post] (v1');
 \node [transition] (c) [below of = v1'] {$e_{2}$}
   edge [pre] (v1');
\end{scope}
  \begin{scope}[xshift=4.25cm]
    \node(0,0) {$=$};
  \end{scope}
\begin{scope}[xshift=6cm]
 \node [place, label = left:$b_{11}$] (w0)  {};
 \node[place] (w1) [right of=w0,xshift=3cm,label=right:$b_{21}$] {};
 \node [transition] (a*) [above of = w0] {$(e,*)$}
   edge [pre, bend left] (w0)
   edge [post, bend right] (w0);
 \node[transition] (*b) [above of= w1] {$(*,e_{1})$}
   edge [post] (w1);
 \node [transition] (ab) [left of=*b,xshift=-1cm] {$(e,e_{1})$}
   edge [pre,bend angle = 15, bend left] (w0)
   edge [post,bend angle = 15, bend right] (w0)
   edge [post,bend angle = 15, bend left] (w1);
 \node [transition] (*c) [below of=w1] {$(*,e_{2})$}
   edge [pre] (w1);
 \node [transition] (ac) [left of=*c,xshift=-1cm] {$(e,e_{2})$}
   edge [pre,bend angle = 15, bend right] (w1)
   edge [pre,bend angle = 15, bend left] (w0)
   edge [post,bend angle = 15, bend right] (w0);
 \end{scope}
\end{tikzpicture}
\end{center}
\caption{The product of two Petri nets is shown.}
\end{figure}
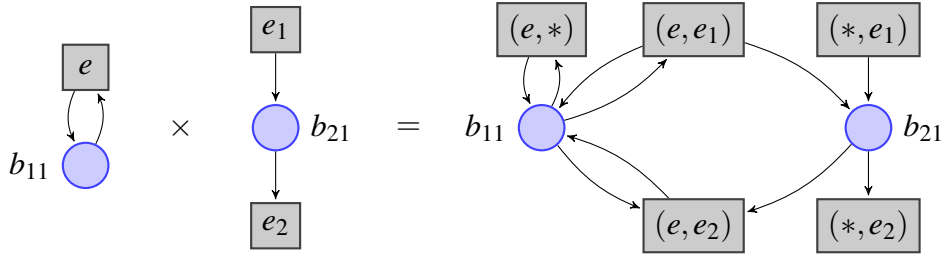

\begin{definition}[{\bf Product of Petri nets}]\label{parallproduct} Let $N_1 =(B_1,E_{1*}, \mathsf{pre}_1,\mathsf{post}_1)$ $N_2 =(B_2,E_{2*}, \mathsf{pre}_2,\mathsf{post}_2)$ be Petri nets. {\it Their product} $N = N_1 \times N_2:= (B,E_*,\mathsf{pre},\mathsf{post});$ it has the events $E_* : = E_{1*} \times E_{2*}$, the product in $\mathbf{Set}_*$ with the projections $\pi_{1*}:E_* \to E_{1*}$ and $\pi_{2*}:E_* \to E_{2*}$. Its conditions have the form $B: = B_1 \sqcup B_2$, the disjoint union of $B_1$ and $B_2$. Define $\rho_1$ to be the opposite relation to the injection $(\rho_1)^{\mathrm{op}}:B_1 \to B$. Define $\rho_2$ similarly. Define the pre and post conditions of an event $e$ in the product in terms of its pre and post conditions in the components by
\begin{align*}
& \mathsf{pre}(e): = (\rho_1)^{\mathrm{op}}[\mathsf{pre}_1(\pi_1(e))] + (\rho_2)^{\mathrm{op}}[\mathsf{pre}_2(\pi_2(e))]\\
& \mathsf{post}(e): = (\rho_1)^{\mathrm{op}}[\mathsf{post}_1(\pi_1(e))] +  (\rho_2)^{\mathrm{op}}[\mathsf{post}_2(\pi_2(e))].
\end{align*}

If either $N_1 \ne N$ or $N_2 \ne N$, then we say that $N_1$ is \textit{the factors of $N$} and $N$ is {\it decomposable}. Otherwise, we say that $N$ is {\it undecomposable}.
\end{definition}

As we have seen any Petri net can be considered as a bipartite directed graphs, hence all above results are true for Petri net. In particular, we get the following criteria of decomposition (=parallalization) of Petri nets.

\begin{theorem}\label{theorem}
Let $(N,\varphi)=\bigl((B,E_*, \mathsf{pre},\mathsf{post}),\varphi\bigr)$ be a Petri net with an injection $\varphi:B \to \mathbb{N}$, let $ P_1, P_2 \in \mathbb{N}[x,y]$ such that $\tau(P_1)\cap \tau(P_2) = \varnothing$. Then $P(N,\varphi) = P_1 \cdot P_2$ if and only if $\mathscr{N}(P_1) \times \mathscr{N}(P_2) \cong N$, in the category $\mathbf{PN}$.

In particular, a Petri net $N = (B,E,\mathsf{pre},\mathsf{post})$ is decomposable if and only if the polynomial $P(N,\varphi)$, for an arbitrary choice of an injection $\varphi:B \to \mathbb{N}$, is decomposable over $\mathbb{N}[x,y]$, i.e., $P(N,\varphi) = P_1 \cdot P_2$ for some $P_1,P_2 \in \mathbb{N}[x,y]$, and $\tau(P_1) \cap \tau(P_2) = \varnothing$.
\end{theorem}

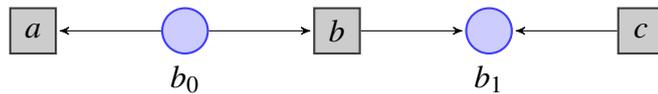
\begin{figure}[h!]
\begin{center}
\begin{tikzpicture}[node distance=1.3cm,>=stealth',bend angle=25,auto]
  \tikzstyle{place}=[circle,thick,draw=blue!75,fill=blue!20,minimum size=6mm]
  \tikzstyle{red place}=[place,draw=red!75,fill=red!20]
  \tikzstyle{transition}=[rectangle,thick,draw=black!75,
  			  fill=black!20,minimum size=6mm]
  \node[place,label=below:$b_0$] (0) at (0,0) {};
  \node[place,label=below:$b_1$] (1) at (4,0) {};
  \node[transition] (a) at (-2,0) {$a$} edge [pre] (0);
  \node[transition] (b) at (2,0) {$b$} edge [pre] (0) edge[post] (1);
  \node[transition] (c) at (6,0) {$c$} edge [post] (1);
  \end{tikzpicture}
\end{center}
\caption{The Petri net $N$.}\label{parall(1)}
\end{figure}

\begin{example}
Let us consider the following Petri net $N$ which is shown in Fig.\ref{parall(1)}.

Set $\varphi(b_0) = 0$ and $\varphi(b_1) = 1$. Then $a \mapsto x^{2^0} = x$, $b \mapsto x^{2^0}y^{2^1} = xy^2,$ $ \mapsto y^{2^1} = y^2$ and $ * \mapsto 1$, and we get $P(N,\varphi) = x + xy^2 + y^2+1$.

It is not hard to see that $P(N,\varphi) = x + xy^2 + y^2+1 = (x+1)(y^2+1)$. Let $P_1 = x+1$ and $P_2 = y^2+1$. It is clear that $\tau(P_1) \cap \tau(P_2) = \varnothing.$ By Construction \ref{con2}, we get the Petri nets $\mathscr{N}(P_1)$ and $\mathscr{N}(P_2)$ (see fig.\ref{parall(2)}). It is easy to see that $N = \mathscr{N}(P_1) \times \mathscr{N}(P_2)$.

\hfill$\square$
\end{example}

\begin{figure}[h!]
\begin{center}
\begin{tikzpicture}[node distance=1.3cm,>=stealth',bend angle=25,auto]
  \tikzstyle{place}=[circle,thick,draw=blue!75,fill=blue!20,minimum size=6mm]
  \tikzstyle{red place}=[place,draw=red!75,fill=red!20]
  \tikzstyle{transition}=[rectangle,thick,draw=black!75,
  			  fill=black!20,minimum size=6mm]
  \node[place,label=right:$0$] (0) at (0,0) {};
  \node[place,label=right:$1$] (1) at (4,0) {};
  \node[transition] (e1) at (0,-1.5) {$e_1$} edge [pre] (0);
  \node[transition] (e2) at (4,-1.5) {$e_2$} edge [post] (1);
  \end{tikzpicture}
\end{center}
\caption{Two Petri nets $\mathscr{N}(P_1)$ (at left) and $\mathscr{N}(P_1)$ (at right) that correspond to the polynomials $P_1 = x+1$ and $P_2 = y^2+1$, respectively.}\label{parall(2)}\end{figure}
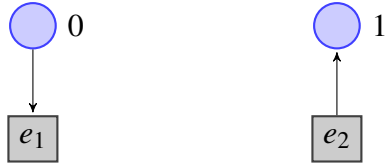

\end{document}